\documentclass[12pt,oneside]{amsart}
\usepackage{t1enc}
\usepackage[latin2]{inputenc}

\usepackage{amsmath,amsfonts}
\usepackage{amsthm}
\usepackage{amscd}
\usepackage{amssymb}
\usepackage{enumerate}
\usepackage{mathrsfs}
\usepackage{dsfont}

\usepackage{stmaryrd}
\usepackage[T1]{fontenc}

\usepackage{mathtools}

\usepackage{bbm}

\usepackage{tikz}
\usepackage{tikz-cd}
\usepackage{wasysym}
\usepackage{verbatim}
\usepackage{fancyhdr}
\usepackage{fancybox}
\usepackage{url}
\usepackage{color}

\usepackage[pdfencoding=auto, psdextra]{hyperref}
\hypersetup{colorlinks=true, citecolor=black, linkcolor=black, 
bookmarksopen=true,
pdftitle={P-measures in Models without P-points},
pdfauthor={Borodulin-Nadzieja and Cancino-Manriquez and Morawski}, pdfpagemode=FullScreen}

\newcommand{\eps}{\varepsilon}

\newcommand{\bc}{\begin{center}}
\newcommand{\ec}{\end{center}}

\newcommand{\I}{\mathcal{I}}

\newcommand{\F}{\mathcal{F}}

\newcommand{\SSi}{\mathbb{S}i}
\newcommand{\forc}{\Vdash}

\newcommand{\Anp}[1]{\left\langle\,#1\,\right\rangle}
\newcommand{\anp}[1]{\langle\, #1 \,\rangle}

\newcommand{\set}[1]{\left\{\, #1\,\right\}}

\DeclareMathOperator{\Dom}{Dom}
\DeclareMathOperator{\Cod}{Cod}

\newtheorem{thm}{Theorem}[section]
\newtheorem{lem}[thm]{Lemma}
\newtheorem{prop}[thm]{Proposition}
\newtheorem{cor}[thm]{Corollary}
\newtheorem{fact}[thm]{Fact}

\theoremstyle{definition}

\newtheorem{prob}[thm]{Problem}

\newtheorem{df}[thm]{Definition}
\newtheorem{exa}[thm]{Example}

\newtheorem{rem}[thm]{Remark}

\title[P-measures in models without P-points]{P-measures in models without P-points}

\author{Piotr Borodulin--Nadzieja}
\address[Piotr Borodulin-Nadzieja]{Mathematical Institute, University of Wroc\l aw \\   pl. Grunwaldzki 2, 50-384 Wroc\l aw, Poland}
\email{pborod@math.uni.wroc.pl}

\author{Jonathan Cancino-Manr\'{i}quez}
\address[Jonathan Cancino-Manr\'{i}quez]{Czech Academy of Sciences \\ \v{Z}itn'{a} 25, Praha 1, Czech Republic}
\email{cancino@math.cas.cz}

\author{Adam Morawski}
\address[Adam Morawski]{Mathematical Institute, University of Wroc\l aw \\   pl. Grunwaldzki 2, 50-384 Wroc\l aw, Poland}
\email{addam.morawski@gmail.com}

\thanks{The first author was supported by the 
 Polish National Science Center under the Weave-UNISONO call in the
Weave programme, no. 2021/03/Y/ST1/00124. \\
\indent The second author was supported by the "Programme to support
prospective human resources -- post Ph.D. candidates" of the Czech
Academy of Sciences, project L100192251.}


\begin{document}

\begin{abstract} We answer in negative the problem if the existence of a P-measure implies the existence of a P-point. Namely, we show that if we add random reals to a certain 'unique P-point' model, then in the resulting model we will have a
	P-measure but not P-points. Also, we investigate the question if there is a P-measure in the Silver model. We show that rapid filters cannot be extended to a P-measure in the extension by $\omega$ product of Silver forcings and that in the model
	obtained by the countable support $\omega_2$-iteration of countable product of Silver forcings there are no P-measures of countable Maharam type.
\end{abstract}

\maketitle

\section*{Introduction}
We call an ultrafilter $\mathcal{U}$ on $\omega$ a P-point if every decreasing sequence of its elements has a pseudointersection in $\mathcal{U}$. Notice that non-principal ultrafilters on $\omega$ cannot be closed under countable intersections. In a sense, a non-principal
P-point is an ultrafilter as close to being closed under countable intersections as possible. There is a huge literature about P-points: they proved their importance in infinitary combinatorics, topology, and forcing.

The existence of P-points is independent of the usual axioms of set theory. They exist under the Continuum Hypothesis and many other axioms. However, Shelah showed that consistently there are no P-points (his proof can be now found in
\cite{shelppoint}). Quite recently Chodounsk\'y
and Guzm\'an \cite{chod}
showed that there are no P-points in the classical Silver model. It is still unknown if P-points exist in the classical random model (see also remarks at the end of the article).

In a similar fashion to P-points one can define P-measures (note that they are usually called 'measures with additive property' in the literature): say that a (finitely additive) measure on $\omega$ is a P-measure if
every decreasing sequence of subsets of $\omega$ has a pseudointersection whose measure is equal to the limit of measures of the elements of the sequence. If a measure on $\omega$ vanishes on points, then it cannot be $\sigma$-additive but P-measures are as
close to being $\sigma$-additive as possible.

If $\mathcal{U}$ is a P-point, then the measure $\delta_\mathcal{U}$, i.e. the 0-1 measure concentrated at $\mathcal{U}$, is a P-measure. It is also known that if there is a P-point, then there is an atomless P-measure (see \cite{blassple}) but the following problem remained open.

\begin{prob}\label{main_problem} Does the existence of P-measures imply the existence of P-points?
\end{prob}

This problem was investigated e.g. in \cite{mekl}, \cite{blassple}, \cite{grebik}. In \cite{mekl} Mekler showed that in Shelah's model witnessing the lack of P-points, there are no P-measures. In the light of Chodounsk\'y-Guzm\'an theorem, it is natural to ask the following.

\begin{prob}\label{silver_question} Is there a P-measure in the Silver model?
\end{prob}

Problem \ref{silver_question} was the initial motivation for our article and we hoped that the solution would either bring us the full answer to Problem \ref{main_problem} (in case of the positive answer) or at least it would be a strong indication that
Problem \ref{main_problem} has a negative solution. Ironically, we have not been able to solve Problem \ref{silver_question} but we have solved Problem \ref{main_problem}.

In \cite{p_measures_random} it is proved that every P-point can be extended to a P-measure in the random extensions. The first author believed that
this theorem can be used to prove that there is a P-point in the classical random model, if one manages to answer Problem \ref{main_problem} in positive. To his surprise, the second author used it to answer Problem \ref{main_problem} in
negative.

\begin{thm}\label{mmmain} It is consistent with $\mathsf{ZFC}$ that there is a P-measure but there is no P-point.
\end{thm}

The main idea behind the proof is the following. A well-known theorem of Kunen says that no selective ultrafilter can be extended to a P-point in any random extension. Shelah proved that there is a model where all P-points are selective. The
model witnessing the negative answer to Problem \ref{main_problem} is the random extension of a model obtained similarly to Shelah's model. There are P-measures there because of the theorem mentioned in the previous paragraph. The proof that
there are no P-points required an upgrade of both Kunen's theorem and Shelah's construction.

Theorem \ref{mmmain} made Problem \ref{silver_question} slightly less compelling. However, we think that the question about the existence of P-measures in the Silver model is still interesting (and it is still open). Moreover, considerations around Problem \ref{silver_question} led us to some interesting notions in the theory of
measures on $\omega$. 

The natural approach toward a solution of Problem \ref{silver_question} is to analyze which filters from the ground model can be extended to a P-measure in the Silver extension, i.e. for which filters $\mathcal{F}$ there is a measure $\mu$ in the
extension such that $\mu(F)=1$ for each $F\in \mathcal{F}$.  E.g. it is not hard to see that every such filter would have to be a P-filter. In fact, for technical reasons, instead of extensions by a single Silver forcing, we will consider extensions by $\omega$ product of Silver forcings (see also the comments at the end of the introduction). In what follows, by $\SSi$ we will denote the Silver forcing, and by $\SSi_\omega$ the $\omega$ product of Silver forcings.

Our first achievement in this direction was the following.

\begin{thm} No rapid filter can be extended to a P-measure in the extension by the $\omega$ product of Silver forcings.
\end{thm}

One can also analyze which measures from the ground model can be extended to a P-measure in the extension (again such measures would also have to be P-measures). 
As we have already pointed out, the basic examples of measures are those concentrated on an ultrafilter. Ultrafilters can be used to define also non-atomic measures, so-called
ultrafilter densities (see Section \ref{measures-on-omega} for the precise definitions). We prove that no such measure can be extended to a P-measure in the model $V^{\SSi_\omega}$. 

\begin{thm}\label{Rudy_Blass} If there is an ultrafilter $\mathcal{U}$ such that $\delta_\mathcal{U}$ is Rudin-Blass below a measure $\mu$, then $\mu$ cannot be extended to a P-measure in the model obtained by $\omega_2$ iterations of $\SSi_\omega$. In particular, no ultrafilter and no ultrafilter density can be extended to a P-measure in this model. 
\end{thm}

When we proved the above theorem, in fact, we did not know any example of a P-measure (under $\mathsf{CH}$) that does not satisfy its assumptions, so we did not have any candidate for a which can be extended to a P-measure in the $V^{\SSi_\omega}$
Then, we managed to show that under $\mathsf{CH}$ there is a P-measure $\lambda$ which is in a sense the Lebesgue measure on $2^\omega$ in disguise. This measure is basically different from all the examples of P-measures known before: it is non-atomic but its Maharam type is countable. Recall that the Maharam type of a measure $\mu$ on $\omega$ is the density of the pseudo-metric on $\mathcal{P}(\omega)$ defined by $d_\mu(A, B) = \mu(A \triangle B)$. 

Additionally, we constructed $\lambda$ in a way that it does not satisfy the assumption of  Theorem \ref{Rudy_Blass}, so it looked like a promising candidate for a measure that can be extended to a P-measure in the model $V^{\SSi_\omega}$. However, we
managed to upgrade the above theorem (to Theorem \ref{Nearly-no}) so that as a corollary we get the following.

\begin{thm}\label{mahmah} In the model obtained by the countable support $\omega_2$-iteration of countable products of the Silver forcings there are no P-measures of countable Maharam type.
\end{thm}

Since P-points are in a sense measures of Maharam type $2$, the above theorem can be seen as a generalization of the Chodounsk\'y-Guzm\'an theorem. It also says that $\lambda$, as a measure of countable Maharam type, cannot be extended to a P-measure in
the Silver extension.

Finally, we are again in a situation in which we do not know any example of a measure that has a chance to be extended to a P-measure in the model $V^{\SSi_\omega}$.  It might mean that there are no P-measures in the model obtained by $\omega_2$ iteration of $\SSi_\omega$. Or there are still some interesting P-measures to discover in the models satisfying $\mathsf{CH}$.

We will comment on the model which appeared in Theorem \ref{mahmah}. This is not the classical Silver model (that is obtained by the countable support iteration of $\omega_2$ Silver forcings). However, as can be seen e.g. in \cite{chod}, the $\omega$
product of Silver forcings is much easier to work with, in this context, than the single Silver forcing. On the other hand, if we want to consider the uncountable product of Silver forcings, then we will face problems trying to deal with the
captured names for objects from the final model at the intermediate steps. So, it seems that the 'mixture' of iteration and products is the most convenient working environment here. We conjecture that those problems are purely technical and that the assertion of Theorem \ref{mahmah}
is true as well for the classical Silver model and the model obtained by the side-by-side product of Silver models. There is a general and somewhat vague question of how different are those three versions of the Silver model and it seems that not much is known about
the answer.

\tableofcontents

\section{P-points and P-measures}

We start by recalling the basic definitions.
    
A {\it filter} over a set $X$ is a family of subsets of $X$ closed under intersections and supersets. We will consider only proper filters, i.e. those which do not contain $\emptyset$. An {\it ultrafilter} is a $\subseteq$-maximal proper filter or,  equivalently, a filter such that for every $A \subseteq X$ either it contains $A$ or it contains $A^c$.
An ultrafilter is called {\it non-principal} if it does not contain any singleton (and thus any finite set). 

A {\it finitely additive measure} on $X$ is a function $\mu\colon\mathcal{P}(X)\to \mathbb{R}_{\ge0}$ such that 
    \begin{enumerate}
        \item $\mu(\emptyset)=0$,
        \item for $A\cap B=\emptyset$ we have $\mu(A\cup B)=\mu(A)+\mu(B)$.
    \end{enumerate}
    We say that a measure {\it vanishes on points} if no singleton has a positive measure (and so all finite sets have measure 0).
    A {\it probability} measure is a measure with $\mu(X)=1$.
    
    A probability measure is said to be atomless if for any $\eps>0$ the set $X$ can be divided into finitely many sets of measure not bigger than $\eps$.

Further on, in this paper all ultrafilters will be non-principal ultrafilters on $\omega$ and all measures will be vanishing on points, finitely additive probability measures on $\omega$.

We say that $A$ is {\it almost} contained in $B$ (denoted $A\subseteq^*B$) if $A\setminus B$ is finite. For a family $\F$ we say that $Z$ is a {\it pseudointersection} of $\F$ if $Z\subseteq^* F$ for each $F\in \mathcal{F}$. 
Notice that $Z$ is a pseudointersection of a countable family $\{F_n\colon n\in\omega\} \subseteq \mathcal{P}(\omega)$ if and only if there is a function  $f\colon \omega\to \omega$ such that $Z\setminus f(n)\subseteq F_n$ for each $n$.

Finally, the following two notions are of main interest. An ultrafilter $\mathcal{U}$ is called a \emph{P-point} if for every decreasing sequence 
$\{ U_n\colon n\in\omega\}$ of elements of $\mathcal{U}$ there is $Z\in \mathcal{U}$ which is a pseudointersection of $\{
	U_n\colon n\in\omega\}$.
    A measure $\mu$ is called a \textit{P-measure} if for every decreasing sequence $\{A_n\colon n\in\omega\}$ there is $Z$, a pseudointersection of $\{A_n\colon n\in\omega\}$, with $\mu(Z)=\inf \mu(A_n)$ ($=\lim \mu(A_n) $). Such measures are also known as
	measures with the {\it additive property}(*) or AP-measures.

We will formulate now two simple facts concerning P-measures leaving the proofs for the reader.

\begin{prop}
\label{pmeasuredefs}
    The following are equivalent:
        \begin{enumerate}
            \item $\mu$ is a P-measure,
            \item for every increasing sequence $(B_n)$ there is a set $Z$ with $\mu(Z)=\sup_n(B_n)$ and $Z\supseteq^* B_n$ for each $n$,
            \item for every pairwise disjoint family $(C_n)$ there is $Z$ st. $\mu(Z)=\sum_n\mu(C_n)$ and $Z\supseteq^* C_n$ for each $n$.
        \end{enumerate}
\end{prop}

\begin{prop}
    A measure $\mu$ is a P-measure if and only if for any decreasing sequence $(A_n)$ there is an increasing function $f\colon\omega\to \omega$ such that 
    $$
        \mu\left(\bigcup_{n\in\omega}\left(A_n\cap [f(n),f(n+1))\,\right)\right)= \lim_{n\to \infty}\mu(A_n).
    $$
    \label{pmeasfunc}
\end{prop}
    
Proposition \ref{pmeasuredefs} (3) justifies the statement that P-measures are, in some sense, as close to being $\sigma$-additive as a (vanishing on points) measure possibly can.
Proposition \ref{pmeasfunc} carries a valuable intuition that, for P-points and P-measures, any pseudointersection is tightly associated with a function. 

\section{Measures on \ensuremath{\omega}}\label{measures-on-omega}

Note that every (finitely additive) measure $\mu$ on $\omega$ can be seen as a measure on the clopen subsets of $\omega^*$ (as, according to our promise from the previous section, $\mu$ vanishes on points) and then extended uniquely to a Borel (countably additive) measure on $\omega^*$. So, measures on
$\omega$ can be identified with countably additive measures on $\omega^*$.

We will give now some examples of measures on $\omega$.

\begin{exa} Let $\mathcal{U}$ be an ultrafilter on $\omega$. We may produce a measure by assigning value 1 to all sets in $\mathcal{U}$ and value 0 to their complements. We will denote such measure by $\delta_\mathcal{U}$ (and call it \emph{the Dirac delta
	at} $\mathcal{U}$). Clearly, every Dirac delta is atomic.  
\end{exa}

There are many atomless measures on $\omega^*$, as $\omega^*$ is not scattered (see \cite{Rudin}). Here we present a quite general way to produce atomless measures extending the asymptotic density. 

\begin{exa}\label{example2} Let $\mu$ be a measure on $\omega$. Let $d_A\colon \omega \to [0,1]$ be defined by $d_A(n) = |A\cap n|/n$. Let 
	\[ \nu(A) = \int d_A \ d\mu. \]
Formally, the above expression does not make much sense, so let us comment on how to interpret it. First, lift $\mu$ to a Borel measure $\mu'$ on $\omega^*$. The function $d_A$ is bounded so it induces a continuous function $d'_A\colon \omega^* \to [0,1]$.
	Let $\nu'(A) = \int d'_A \ d\mu'$. This is a (countably additive) measure on $\omega^*$ and so it generates a measure $\nu$ on $\omega$. Such a measure extends the asymptotic density.\footnote{Note that it is also possible to define the Lebesgue integral directly for finitely additive measures.} 
\end{exa}

We will consider now measures like the above of quite a specific form.

\begin{exa}\label{extension-density}
	Suppose that $\mu$ is a Dirac delta concentrated at an ultrafilter $\mathcal{U}$. Let $d_A$ be defined as in Example \ref{example2}. Then 
	\[ \nu(A) = \int d_A \ d\mu =  \lim_{n\to \mathcal{U}} d_A(n). \]
	This is an atomless measure extending the asymptotic density. We will call such measures \emph{ultrafilter densities}.
	Note that if $\mathcal{U}$ is a P-point, then the resulting measure is a P-measure (see e.g. \cite{blassple}).
\end{exa}

Such measures were studied e.g. in \cite{mekl} or \cite{blassple}. Notice that two different ultrafilters may induce (in the sense of Example \ref{extension-density}) the same measure (see \cite{grebik}, \cite{Kunisada}).

It seems to be not easy to define a measure on $\omega$ 'directly' (even using ultrafilter as a parameter) in a basically different way than in the above examples. In Section \ref{themeasure} we will present one more example, which is defined
without any (direct) use of ultrafilters.

Now we will investigate relations between measures on $\omega$. Recall the definition of Rudin-Blass ordering of filters.

\begin{df} Suppose $\mathcal{F}$, $\mathcal{G}$ are filters on $\omega$. We say that $\mathcal{G}$ is \emph{Rudin-Keisler reducible to} $\mathcal{F}$, $\mathcal{G} \leq_{RK} \mathcal{F}$ in short, if there is a function $f\colon \omega
	\to \omega$ such that $X\in \mathcal{G}$ if and only if $f^{-1}[X] \in \mathcal{F}$.
We say that $\mathcal{F}$ is \emph{Rudin-Blass reducible to} $\mathcal{G}$ ($\leq_{RB}$) if there is such finite-to-one function.
\end{df}

We say that a filter $\mathcal{F}$ is \emph{nearly ultra} if there is an ultrafilter $\mathcal{U}$ such that $\mathcal{U} \leq_{RK} \mathcal{F}$.

The definition of Rudin-Keisler ordering can be naturally re-used to introduce an ordering on measures on $\omega$.

\begin{df} Suppose $\mu$, $\nu$ are measures on $\omega$. We say that $\nu$ is \emph{Rudin-Keisler reducible to} $\mu$ if there is a function $f\colon \omega\to \omega$ such that for every $X\subseteq
	\omega$ we have $$\nu(X) = \mu(f^{-1}[X]).$$
We say that $\nu$ is \emph{Rudin-Blass reducible to} $\mu$ if there is such a finite-to-one function.
\end{df}

Similarly, we will say that a measure $\mu$ is \emph{nearly Dirac} if there is an ultrafilter $\mathcal{U}$ such that  $\delta_\mathcal{U} \leq_{RK} \mu$. Notice that a measure is nearly Dirac if and only if its filter of measure 1 sets is nearly ultra.

The following fact will allow us to use both Rudin-Keisler and Rudin-Blass orderings interchangeably.
\begin{fact}\label{rk-is-rb}
    If $\mathcal{G}$ is a P-filter and $\mathcal{F}$ contains the Fr\'echet filter then 
    \[\mathcal{F}\leq_{RK}\mathcal{G} \iff \mathcal{F}\leq_{RB}\mathcal{G}.\]
    If $\mu$ is a P-measure and $\nu$ is vanishing on points then
    \[\nu \leq_{RK} \mu \iff \nu \leq_{RB} \mu.\]
\end{fact}
\begin{proof}
    Right-to-left direction is immediate. For the other side, we provide a proof in the case of measures. 

    Take a P-measure $\mu$ and $\nu\leq_{RK}\mu$. Let $f\in\omega^\omega$ be a witnessing function. Since $\nu$ vanishes on points we know that $f$ is $\mu$-null-to-one. Let $A_n=f^{-1}[\omega\setminus n ]$. Then $\mu(A_ n)=1$ and there is a set $A$, a pseudointersection  of $(A_n)_n$ with $\mu(A)=1$. But for any $n\in\omega$ we have $A\cap f^{-1}[n]$ finite (as it is disjoint from $A_{n+1}$). Hence $f$ is finite-to-one on a set of full measure $\mu$. It is now easy to alter $f$ on $\omega\setminus A$ to get a finite-to-one function witnessing $\nu\leq_{RB}\mu$.
\end{proof}
\begin{prop}\label{ult-extend-meas-nd} Every measure that is an ultrafilter density is nearly Dirac.
\end{prop}
\begin{proof} Let $\mu$ be an ultrafilter extension of density. Fix an ultrafilter $\mathcal{U}$ such that $\mu(A) = \lim_{n\to \mathcal{U}} |A \cap n|/n$.

	Let $g\colon \omega \to \omega$ be an increasing function such that $g(k+1)/g(k) \to \infty$. Define $L_i = \bigcup_k [g(2k+i), g(2k+i+1))$ for $i\in \{0,1\}$ and notice that $L_0, L_1$ is a partition of $\omega$ so exactly one of them has to be in
	$\mathcal{U}$. Assume, for simplicity, that $L_1 \in \mathcal{U}$; the other case is analogous.

	Let $f\colon \omega \to \omega$ be defined by $f(n)=k$ if $n\in [g(2k), g(2k+2))$. 
	Define \[ \mathcal{V} = \{X\subseteq \omega\colon f^{-1}[X]\in \mathcal{U}\} \] and notice that $\mathcal{V}$ is an ultrafilter.
	We will show that $f$ witnesses that $\delta_\mathcal{V} \leq_{RB} \mu$.

	Let $X\in \mathcal{V}$. Fix $\varepsilon>0$. 
	Denote
	\[ A_\varepsilon = \{ n\colon |\frac{|f^{-1}[X]\cap n|}{n} - 1| < \varepsilon \}. \]
	If $n\in [g(2k+1), g(2k+2))$, then 
	\[ |f^{-1}[X]\cap n| \geq n - g(2k) \geq g(2k+1)-g(2k) > (1-\varepsilon) n \]
	for $k$ big enough (since $n\geq g(2k+1)$ and $g(2k)/g(2k+1)<\varepsilon$ for $k$ big enough).

	It follows that $f^{-1}[X] \cap L_1 \subseteq^* A_\varepsilon$. But $f^{-1}[X] \in \mathcal{U}$ (as $X\in \mathcal{V}$) and $L_1 \in \mathcal{U}$. Thus, $A_\varepsilon\in \mathcal{U}$ and, since $\varepsilon$ was arbitrary, $\mu(X) = 1$.

	Analogously, if $X\notin \mathcal{V}$, then $X^c\in \mathcal{V}$, $\mu(f^{-1}[X]^c)=1$ and so $\mu(f^{-1}[X])=0$. Hence, $\delta_\mathcal{V} \leq_{RB} \mu$.
\end{proof}

\begin{prop}\label{P-hered}
    Being a P-filter or a P-measure is hereditary with respect to Rudin-Keisler's ordering.
\end{prop}
\begin{proof}
	We only provide a short proof for the case of P-measures. Consider a P-measure $\mu$, a function $g\colon \omega \to \omega$ and an $\subseteq^*$-decreasing sequence $(A_n)_n$. By \ref{rk-is-rb} it is enough to consider $g$ finite-to-one. Let $B$ be a pseudointersection of the sequence
	$(g^{-1}[A_n])_n$ suitable for the measure $\mu$. Since $\mu(g^{-1}[g[B]]) \geq \mu(B)$, for all $n\in\omega$, $g^{-1}[g[B]]\subseteq^* A_n$, and we know that $g[B]$ is a pseudointersection of $(A_n)_n$ suitable for $\mu g^{-1}$.
\end{proof}

Hence the ultrafilter $\mathcal{V}$ from the proof of Proposition \ref{ult-extend-meas-nd} is a P-point whenever $\mu$ is a P-measure (even if the ultrafilter used in the definition of $\mu$ is not a
P-point). So we have the following corollary (which was proved also by Grebik in \cite{grebik}).

\begin{cor} If there is a P-measure, which is an ultrafilter density, then there is a P-point.
\end{cor}

Perhaps it is quite surprising that consistently all measures on $\omega$ (vanishing on points) are nearly Dirac. Recall that the principle of filter dichotomy implies that all filters that are not meager are nearly ultra. But the filters of measure
$1$ sets are always non-meager:

\begin{prop}\label{Daria} If $\mu$ is a measure on $\omega$, then the filter $\mathcal{F} = \{F\colon \mu(F)=1\}$ is non-meager.
\end{prop}
\begin{proof} Suppose that $\mathcal{F}$ is meager. Then, by Talagrand's characterization of meager filters (see e.g. \cite{Blass}), there is an interval partition $(I_n)$ such that no infinite union of $I_n$'s is in the dual ideal. Let $(A_\alpha)_{\alpha<\omega_1}$ be an almost
	disjoint family (of subsets of $\omega$) and for $\alpha<\omega_1$ let $N_\alpha = \bigcup_{n\in A_\alpha} I_n$. Then $\mu(A_\alpha)>0$ for each $\alpha$ and $\mu(A_\alpha \cap A_\beta)=0$ if $\alpha \ne \beta$, a contradiction. 
\end{proof}

Notice that if $\mathcal{U}\leq_{RK} \mathcal{F}$ for a filter of the form $\{F\colon \mu(F)=1\}$ for some measure $\mu$ and an ultrafilter $\mathcal{U}$, then $\delta_\mathcal{U} \leq_{RK} \mu$. So, we have the following. 

\begin{prop} Under the principle of filter dichotomy, all measures on $\omega$ are nearly Dirac.
\end{prop}

However, it is not true that all measures on $\omega$ are nearly Dirac in $\mathsf{ZFC}$. We have the following example.

\begin{prop}\label{near} In the random model there is a measure on $\omega$ which is not nearly Dirac.
\end{prop}
\begin{proof} In \cite[Theorem 5.6]{p_measures_random} the authors showed that in the random model there is a measure supported by a filter $\mathcal{F}$ which is generated by a family $\mathcal{A}$ of size $\omega_1$. Suppose that $\mathcal{G}$ is a filter such that $\mathcal{G} \leq_{RB} \mathcal{F}$. Then $\mathcal{G}$ is generated by at most $\omega_1$ sets
	(namely, $\{f[A]\colon A\in \mathcal{A}\}$, where $f$ is a witness for $\mathcal{G} \leq_{RB}\mathcal{F}$). But, in the random model the ultrafilter number $\mathfrak{u} = \omega_2$, so $\mathcal{G}$ cannot be an ultrafilter.
\end{proof}

In Section \ref{themeasure} we will present another example of a measure that is not nearly Dirac, this time under $\mathsf{CH}$.

\section{Silver forcing}
By $\SSi$ we will denote the set of partial functions from $\omega$ to $\{0,1\}$ with co-infinite domains ordered by the reverse inclusion. Such partial order is called the Silver forcing. We will use standard forcing notation. In particular, let $V$
be the ground model, let $G$ be a $\SSi$-generic over $V$. We will use the following
abbreviations and provide the necessary context to conclude whether an object is in $V$ or $V[G]$:
\begin{itemize}
    \item instead of $\dot x[G]$, the evaluation of a name by a generic, we will simply write $x$,
    \item instead of $\check x$, the standard name for $x \in V$, we will also write $x$.
\end{itemize}

\subsection{$\omega^\omega$-bounding.}

We will need a few properties of the Silver forcing. The following facts hold for $\SSi$.
\begin{prop}[$\omega^\omega$-bounding]
\label{omom}
    If $f\in V[G]$ is a function from $\omega$ to $\omega$, then there is a function $g\in V$ with $\forall n\ g(n)\geqslant f(n)$.
\end{prop}
A proof of Proposition \ref{omom} can be found in \cite{halb12}. 

In what follows, instead of $\SSi$ we will often consider the $\omega$ (side-by-side) product of $\SSi$. We will denote it by $\SSi_\omega$. Recall that the forcing $\SSi_\omega$ is $\omega^\omega$-bounding, too.

We can extend our definition of the $\omega^\omega$-bounding property to the relation between models. We say that $W\supseteq V$ has the $\omega^\omega$-bounding property (over $V$) if for any $f\colon \omega\to \omega$ form $W$ there is $g\colon \omega\to \omega$ from $V$ with $f(n)\leqslant g(n)$. 

\begin{prop} \label{P-needed}
    For \,$W\supseteq V$ (models of ZFC) and $W$ having the $\omega^\omega$-bounding property over $V$, if $\mu\in V$ is a measure that can be extended to a P-measure in $W$, then $\mu$ is a P-measure.
    \label{omompmeas}
\end{prop}
\begin{proof}
    Take any decreasing sequence $\anp{A_n:n\in\omega}\in V$. If $\mu'$ is a P-measure extending $\mu$ in $V[G]$ then there is some function $f\in V[G]$ such that, if
    $$
    A=\bigcup_{n\in\omega}\left(A_n\cap [f(n),f(n+1))\,\right),
    $$
    then $\mu'(A)=\lim_{n}\mu'(A_n)$, by Proposition \ref{pmeasfunc}.
    
    Now by $\omega^\omega$-bounding property we can take $\tilde f\in V$ with $\tilde f(n)\geqslant f(n)$. 
    
    If we put 
    $$A'=\bigcup_{n\in\omega}\left(U_n\cap [\tilde f(n),\tilde f(n+1))\,\right),$$ 
	then $A'\supseteq A$ and $A'$ is still a pseudointersection of $\set{A_n:n\in\omega}$ and \[ \lim_n\mu'(A_n) \geq \mu'(A') \geq \mu'(A) = \lim_n \mu'(A_n)=\lim_n\mu(A_n). \]
    
    But now $A'\in V$ and so $\mu(A')=\mu'(A')=\lim_n\mu(A_n)$.
\end{proof}
As we know that a P-point can be viewed as a P-measure we get the following remark:
\begin{rem}
	For \,$W\supseteq V$ and $W$ having the $\omega^\omega$-bounding property over $V$, if $\mu\in V$ is a measure that can be extended to a P-measure in $W$, then $\mu$ is a P-measure.
    \label{omomppoint}
\end{rem}

The following proposition allows us to perform 'reflection' arguments for measures.

\begin{prop}\label{extending} 
	Let $\mathbb{P}=\langle \mathbb{P}_\alpha,\dot{\mathbb{Q}}_\alpha:\alpha\in\omega_2\rangle$ be a countable support iteration of proper forcings such that for each $\alpha\in\omega_2$, $\mathbb{P}_\alpha\Vdash\vert\dot{\mathbb{Q}}\vert=\omega_1$. Let
	$\dot{\mu}$ be a $\mathbb{P}$-name for a measure on $\omega$. Then there is $\alpha<\omega_2$ and a $\mathbb{P}_\alpha$-name $\dot{\mu}_\alpha$ such that
\begin{equation*}
    \mathbb{P}_\alpha\Vdash\dot{\mu}_\alpha=\dot{\mu}\cap V[G_\alpha]
\end{equation*}
\end{prop}

\begin{proof}
	Let $\mathcal{M}\prec H(\theta)$, where $\theta$ is sufficiently big. Assume that $\mathcal{M}$ is of cardinality $\omega_1$, closed under  countable sequences (i.e. $\mathcal{M}^\omega \subseteq \mathcal{M}$) and such that $\omega_1,\omega_2,\mathbb{P}_{\omega_2},\dot{\mu}\in\mathcal{M}$. Let 
	\[ \alpha=\sup(\omega_2\cap \mathcal{M}).\] Note that $\alpha\subseteq\mathcal{M}$. By \cite[Theorem 4.1 in Chapter III]{pif} for each $\beta<\omega_2$, the poset $\mathbb{P}_\beta$ contains a dense subset of cardinality $\omega_1$, so by elementarity, for each
	$\beta < \alpha$ we can find $D_\beta\in\mathcal{M}$ which is a dense subset of $\mathbb{P}_\beta$; since $\omega_1\in\mathcal{M}$ and $\mathcal{M}$ is closed under taking countable sequences, we also have
	$D_\beta\subseteq\mathcal{M}$. Thus, for our purposes we may assume that $\mathbb{P}_\beta\subseteq\mathcal{M}$ for each $\beta < \alpha$. 

	Actually, we may also assume that $\mathbb{P}_\alpha \subseteq \mathcal{M}$, as $\bigcup_{\beta<\alpha}D_\beta$ is a dense subset of $\mathbb{P}_{\alpha}$ of cardinality $\omega_1$ (since the cofinality of $\alpha$ is $\omega_1$ and each
	condition $p\in\mathbb{P}_{\alpha}$ has a countable support, for each $p\in\mathbb{P}_{\alpha}$ there is $\beta<\alpha$ such that $p\in\mathbb{P}_\beta$, and we can assume that $p\in D_\beta$).
\medskip

	\textbf{Claim.} For every $\mathbb{P}_{\alpha}$-name $\dot{x}$ for a subset of $\omega$, and $p\in\mathbb{P}_{\alpha}$ there are countable $\mathbb{P}_\alpha$-names $\dot{y}$ and $\dot{z}$ in $\mathcal{M}$ and $q\in \mathbb{P}_\alpha$, $q\leq p$ such that
\begin{equation}
	q\Vdash  \dot{x}=\dot{y} \land  \dot{\mu}(\dot{x})= \dot{z}.
\end{equation}

	\begin{proof} (of the claim) First notice that for each $\mathbb{P}_{\alpha}$-name $\dot{x}$ for a real, the set of conditions $p\in\mathbb{P}_{\alpha}\cap\mathcal{M}$ for which there is a countable name $\dot{y}$ in $\mathcal{M}$ such that
		$p\Vdash\dot{x}=\dot{y}$ is a dense subset of $\mathbb{P}_{\alpha}$. Indeed, let $p\in \mathbb{P}_\alpha$. That there is $q\leq p$ and a countable name $\dot{y}$ such that $q\Vdash\dot{x}=\dot{y}$ follows from the properness of	$\mathbb{P}_\alpha$.  Since $\dot{y}$ is countable, it only uses countable many conditions from $\mathbb{P}_{\alpha}\subseteq\mathcal{M}$, so we have $\dot{y}\in\mathcal{M}$. 

		Now, let $\dot{x}$ be an arbitrary  $\mathbb{P}_{\alpha}$-name for a subset of $\omega$, and let $\dot{z}_0$ be a $\mathbb{P}_{\omega_2}$-name for the value of $\dot{\mu}(\dot{x})$ (note that a priori, here $\dot{z}_0$ is a
		$\mathbb{P}_{\omega_2}$-name even though $\dot{x}$ is a $\mathbb{P}_{\alpha}$-name). Let $p\in\mathbb{P}_{\alpha}$, let $\dot{y}\in\mathcal{M}$ be a countable name such that there is $q\leq p$ such that $q\Vdash\dot{x}=\dot{y}$. Since
		$\dot{y}\in\mathcal{M}$, $q\Vdash\dot{y}\subseteq\omega$, and $\dot{\mu}\in\mathcal{M}$, by elementarity, there is a $\mathbb{P}_{\omega_2}$-name $\dot{z}_1\in \mathcal{M}$ for $\dot{z}_0$. Note that $q\Vdash\dot{z}_0 =\dot{z}_1$. We can
		think of $\dot{z}_{1}$ as a name for a sequence of rationals converging to the actual value of $\dot{\mu}(\dot{y})$, so we can find $q'\leq q$ and a countable name $\dot{z}$ for $\dot{z}_1$, and by elementarity we can find both $q'$ and $\dot{z}$ in $\mathcal{M}$. Now, $\mathcal{M}$ knows that $\dot{z}$ is countable, so $\dot{z}\subseteq\mathcal{M}$, which implies that any condition appearing in $\dot{z}$ should be an element from $\mathcal{M}$, so $\dot{z}$ is actually a $\mathbb{P}_{\alpha}$-name.  
	\end{proof}

		Now we define a $\mathbb{P}_{\alpha}$-name as follows:
\begin{enumerate}
     \item[i)] For every  $\mathbb{P}_{\alpha}$-name $\dot{x}$ for a subset of $\omega$ and $p\in\mathbb{P}_{\alpha}$, let $y_{\dot{x},p}$, $z_{\dot{x},p}$ and $q_{\dot{x},p}\in\mathbb{P}_{\alpha}$ be as promised in the claim, i.e.    
     \begin{enumerate}
		\item $y_{\dot{x},p}$, $z_{\dot{x},p}$ are countable $\mathbb{P}_{\alpha}$-names.
        \item $q_{\dot{x},p}\leq p$.
		\item $q_{\dot{x},p}\Vdash (\dot{x}=y_{\dot{x},p})\land(\dot{\mu}(\dot{x})=z_{\dot{x},p})$
    \end{enumerate}
\item[ii)]	For $\dot{x}$ let
	\[ H(\dot{x}) = \{ \big\langle \langle y_{\dot{x},p}, z_{\dot{x},p}\rangle, q_{\dot{x},p}\big\rangle \colon p\in \mathbb{P}_\alpha\}. \]
\item[iii)] Finally, let 
	\[ \dot{\mu}_\alpha = \bigcup \{ H(\dot{x})\colon \dot{x} \mbox{ is a nice }\mathbb{P}_\alpha\mbox{-name for a subset of }\omega\}. \]
	Here we may think of nice names in any reasonable sense, just to avoid barging into a proper class.
\end{enumerate}

We claim that $\dot{\mu}_{\alpha}$ is the required name. Let $\dot{x}$ be a $\mathbb{P}_{\alpha}$-name for a subset of $\omega$ and $p\in\mathbb{P}_{\alpha}$. We may assume $\dot{x}$ is a nice name. Then we have $y_{\dot{x},p}$, $z_{\dot{x},p}$,
$q_{\dot{x},p}$ well defined, which means that $q_{\dot{x},p}\leq p$ and
\begin{equation}    q_{\dot{x},p}\Vdash \langle y_{\dot{x},p},z_{\dot{x},p}\rangle\in\dot{\mu}_{\alpha}.    
\end{equation}
That is,
\begin{equation}    q_{\dot{x},p}\Vdash\dot{\mu}_{\alpha}(y_{\dot{x},p})=z_{\dot{x},p}.
\end{equation}
By i)(c) above we also have that
\begin{equation}    
q_{\dot{x},p}\Vdash\dot{\mu}_{\alpha}(\dot{x})=\dot{\mu}(\dot{x}).
\end{equation}
\end{proof}

\subsection{Interval partitions}
Forcing with $\SSi$ adds a new real. Namely, if $G$ is a generic, then we can define $\dot{s}=\bigcup\{\,p:p\in G\,\}$ (in $V[G]$). We will call this a Silver real, or a generic real.
A name for such real can be easily procured:
$$
\dot s=\{\,\anp{\hat {\langle n,i\rangle}\,,\,\{\,\langle n,i\rangle\,\}}:n\in \omega,i\in 2\,\}
$$
i.e. we use functions with singleton domains to set every element of $s$.

A standard way of creating a set from a real (a 0-1-sequence) is to treat the real as a characteristic function. 
While considering ultrafilters this approach applied to Silver reals yields poor results due to the following fact:
\begin{prop}
If $\mathcal{U}$ is an ultrafilter in $V$ and $s$ is a  Silver real understood as above, then (in $V[G]$) for some $U\in \mathcal{U}$ either $s^{-1}[1] \subseteq U$ or $s^{-1}[1]\cap U=\emptyset$. \end{prop}
\begin{proof}
	Fix an ultrafilter $\mathcal{U}\in V$.
	Let $\dot s$ be a name for Silver real and take any $p$ with $\Dom p\in \mathcal{U}$. Notice that the set of possible $p$'s is dense in $\SSi$ as any partial function (of co-infinite domain) can be extended to have co-infinite domain that is in
	$\mathcal{U}$.

	Now either $p^{-1}[1]\in \mathcal{U}$ and for $U=\Cod p \cup p^{-1}[1]$ we have $p\forc \dot s^{-1}[1] \subseteq U$ or $p^{-1}[0]\in \mathcal{U}$ and taking $U=p^{-1}[0]$ we have $p\forc \dot s^{-1}[1]\cap U=\emptyset$.
\end{proof}

In other words, if we want to construct an ultrafilter in $V[G]$ extending $\mathcal{U}$, we do not have many options for its behavior on $s^{-1}[1]$.

Therefore from this point on, following \cite{chod}, we will rather turn reals into interval partitions. Namely let $pos(q,k)$ denote the index $m$ for which $q(m)=1$ for the $k$-th time (more precisely $pos(q,k)=\min\{\,m:|q^{-1}[1]\cap[0,m]|=k\,\}$) and we set $pos(q,0)=0$. Then we define
$$
 I_k(r)=[pos(q,k),pos(q,k+1))
$$
and $\I(q)=\{I_k(q)\colon k\in\omega\}$. Notice that such definitions do not require $q$ to be a function with the full domain.

\section{Measures which cannot be extended to P-measures}

In this section, we prove the main result in the Silver-like model from Theorem \ref{mahmah}. We will need the following definition. The name ``silver'' comes from the fact that, as we will see, this definition is closely related to the combinatorics of
$\I(q)$, for $q \in \SSi$.

\begin{df}
	Let $\mu$ be a measure on $\omega$. We say that it is \emph{silver} if for each $N\subseteq \omega$, $r>0$ and each interval partition $(I_n)$ there is $\epsilon \in \{-1,1\}^\omega$
	such that 
	\[ \mu\big( \bigcup_{n} N^{\epsilon(n)} \cap I_n  \big)< r \] 
	(Here we follow the convention that that $N^1 = N$ and $N^{-1} = N^c$.)
\end{df}

If $\mu$ extends asymptotic density, then $\mu$ is not silver (which is witnessed e.g. by $N$ being the set of odd numbers, $r=1/2$ and $I_n = [2^n, 2^{n+1})$ for each $n$). Every Dirac
	delta is silver.

\begin{lem}\label{strongly-atomless}
Suppose that $\mu$ is a measure which is silver and that $p\in \SSi$. Then for every $i\in 2$, $r>0$ we may find $q \leq p$
		such that $\mu (D_i(q))>1-r$. 
\end{lem}

\begin{proof}
	Suppose that $\mu$ is silver and so for every interval partition $(I_n)$, $N\subseteq \omega$ and every $r>0$ there is $\epsilon\in \{-1,1\}^\omega$ such that 
	\[ \mu\big( \bigcup_{n} N^{\epsilon(n)} \cap I_n  \big)> 1-r.\]
	Notice that we may additionally assume that and $\epsilon(n)\cdot\epsilon(n+1) = -1$ for infinitely many $n$'s. It follows from the general fact that for every measure $\nu$ on $\omega$ and for every  interval partition $(J_n)$ we may find an infinite set $X\subseteq \omega$ such that $\nu(\bigcup_{n\in X}
	J_n)=0$ (see also Proposition \ref{Daria}).

	Now, fix $i\in 2$, $p\in \SSi$ and $r>0$. Suppose (strengthening $p$ if necessary) that $\mu(\Cod p)=0$, and let $(a_n)$ be the increasing enumeration of $\Cod p$. For each $n\in \omega$ let $I_n = [a_n, a_{n+1})$. Applying the above
		remark for $(I_n)$\footnote{$(I_n)$ may not be a partition but as we deal with a measure vanishing on points, we can forget some initial segment of $\omega$}, $N = D_i(p) \cap \Dom p$ and $r$ we may find
	$\epsilon\in \{-1,1\}^\omega$ such that $\epsilon(n)\cdot\epsilon(n+1)=-1$ for infinitely many $n$'s and 
\[ \mu\big( \bigcup_{n} N^{\epsilon(n)} \cap I_n  \big)> 1-r.\]
Let $q\leq p$ be such that 
\begin{itemize}
	\item $q(a_0) =1 $ iff $\epsilon(0)=-1$
    \item $a_{n+1} \in \Dom q$ iff $\epsilon(n)\epsilon(n+1)=-1$ and in this case let $q(a_{n+1})=1$. 
\end{itemize}
Then $q \in \SSi$ and $\mu(D_i(q))>1-r$.
\end{proof}

\begin{thm}\label{Nearly-no} Suppose $\mu$ is silver. Then $\mu$ cannot be extended to a P-measure in $V^{\SSi_\omega}$.
\end{thm}

	\begin{proof}
		Suppose $\mu$ is silver. Let $p\in \SSi_\omega$. 
	\bigskip

\textbf{Case 1.} There is $q\leq p$ such that 
	\[ q \Vdash \forall n \ \forall \varepsilon>0 \ \exists m>n \ \mu(\bigcup_{k\in [n,m)} \dot{D}_0(\dot{s}_k)) > 1-\varepsilon. \]

	Let $q$ be as above. Using the fact that $\SSi$ is $\omega^\omega$-bounding (to get $h'$ and then $h(n) = {h'}^{(n)}(0) = h'(h'(...(0)...))$ ), we may assume (taking stronger condition if necessary) that there is an increasing function $h\colon \omega \to \omega$ such that \[ q \Vdash \forall n \
		\mu(\bigcup_{k\in [h(n), h(n+1))} \dot{D}(\dot{s}_k)) > 1-2^{-(n+2)}. \] 

		Let \[ \dot{W}_n = \bigcup_{k\in [h(n),h(n+1)) } \dot{D}(\dot{s}_k) \] for each $n$ and let $\dot{Z}_n = \bigcap_{m \leq n} \dot{W}_m$. Notice that
	\[ q\Vdash \dot{Z}_n \mbox{ is }\subseteq-\mbox{decreasing and } \forall n \ \mu(\dot{Z}_n)\geq 1/2. \]
	Suppose that there is a pseudointersection $\dot{Z}$ in $V^{\SSi_\omega}$ of $(\dot{Z}_n)$ such that $\mu(\dot{Z}) = \lim_n \mu(\dot{Z}_n) \geq 1/2$. Again, assume that $q$ forces it and that there is a function $f\colon \omega \to \omega$ in $V$ such
	that 
	\[ q\Vdash \dot{Z}\setminus f(n) \subseteq \dot{Z}_n. \]

		Now, since $\mu$ is silver, subsequently using Lemma \ref{strongly-atomless}, for each $k$ we can find $q'_k \leq q_k$ such that 
		\[ \mu \big( \bigcap_{k< n} D_1(q'_k) \big) > 3/4 \]
		for every $n\in \omega$. Define
		\[ A_n = \bigcap_{k< h(n+1)} D_1(q'_k). \]
	Since $\mu$ is a P-measure, $(A_n)$ is decreasing and $\mu(A_n)>3/4$ for each $n$ we may find a strictly increasing $i\in \omega^\omega \cap V$ such that
		\[ \mu\big( \bigcup_n A_n \cap [i(n), i(n+1)) \big) \geq 3/4. \]
		By taking $f$ and $i$ bigger if necessary we may assume that for each $n$ and for every $k\in [h(n), h(n+1))$ we have $\mathrm{stem}(q'_k) < f(n) <i(n)$.

		For every $n\in \omega$ and $k\in [h(n), h(n+1))$	extend $q'_k$ to $r_k$ in such a way that \[ F_n: = A_n \cap [i(n), i(n+1)) \subseteq D_1(r_k) \cap \mathrm{stem}(r_k). \] It is possible as $F_n \subseteq D_1(q'_k)$ and $\mathrm{stem}(q'_k) <
		f(n) < i(n)$. 

	 Let $r = (r_0, r_1,
		\dots)$. Then 
		\[ r\Vdash \forall n \ F_n \subseteq \bigcap_{k\in [h(n), h(n+1))} \dot{D_1}(\dot{s}_k) \]
	and so 
	\[ r\Vdash \forall n \ F_n \cap \dot{Z}_n  = \emptyset. \]

	Since $f(n) \leq \min F_n$ for each $n$ we have 
	\[ r\Vdash \bigcup_n F_n \cap \dot{Z} = \emptyset \] 
 But $\mu(\bigcup_n F_n)\geq 3/4$, hence $r\Vdash \mu(\dot{Z})\leq 1/4$, a contradiction.
\bigskip

	\textbf{Case 2.} There is $q\leq p$ such that 
	\[ q \Vdash \exists m \ \exists \varepsilon>0 \ \forall n>m \ \mu(\bigcap_{k \in [m,n)} \dot{D}_1 (\dot{s}_k)) > \varepsilon. \]
	
	Without loss of generality, we assume that $m=0$. Let $\dot{Z}_n = \bigcap_{k \leq n} \dot{D}_1(\dot{s}_k)$. Notice that
	\[ q\Vdash \dot{Z}_n \mbox{ is }\subseteq-\mbox{decreasing and } \forall k \ \mu(\dot{Z}_n)\geq \varepsilon. \]

	Suppose that there is a pseudointersection $\dot{Z}$ in $V^{\SSi_\omega}$ of $(\dot{Z}_n)$ such that $\mu(\dot{Z}) = \lim_n \mu(\dot{Z}_n) \geq \varepsilon$. Assume that $q$ forces it and that there is a function $f\colon \omega \to \omega$ in $V$ such
	that 
	\[ q\Vdash \dot{Z}\setminus f(n) \subseteq \dot{Z}_n. \]

		For every $n\in \omega$, using Lemma \ref{strongly-atomless} we can find $q'_n \leq q_n$ such that $\mu(D_0(q'_n))>1-\varepsilon/2^{n+2}$. Let
		\[ A_n = \bigcap_{m<n} D_0(q'_m) . \]
		The sequence $(A_n)$ is $\subseteq$-decreasing and $\mu(A_n)\geq 1-\varepsilon/2$ for each $n$. Since $\mu$ is a P-measure, there is an increasing $i\in \omega^\omega \cap V$ such that
		\[ \mu\big( \bigcup_n A_n \cap [i(n), i(n+1)) \big) \geq 1-\varepsilon/2. \]

		By taking $f$ and $i$ bigger if necessary we may assume that for each $n$ we have $\mathrm{stem}(q'_n) < f(n) <i(n)$.
		For every $n$ extend $q'_n$ to $r_n$ in such a way that \[ F_n: = A_n \cap [i(n), i(n+1)) \subseteq D_0(r_n) \cap \mathrm{stem}(r_n). \] 

	 Let $r = (r_0, r_1, \dots)$. Then 
	 \[ r\Vdash \forall n \ F_n \subseteq D_0(\dot{s}_n) \]
	and so 
	\[ r\Vdash \forall n \ F_n \cap \dot{Z}_n  = \emptyset. \]
	Since $f(n) \leq \min F_n$ for each $n$ we have 
	\[ r\Vdash \bigcup_n F_n \cap \dot{Z} = \emptyset \] 
 But $\mu(\bigcup_n F_n)\geq 1-\varepsilon/2$, hence $r\Vdash \mu(\dot{Z})<\varepsilon/2$, a contradiction.
\end{proof}

We say that a measure $\mu$ is \textit{nearly silver} if there is a (vanishing on points) measure $\nu \leq_{RK} \mu$ such that $\nu$ is silver.

\begin{cor}\label{nonuni} If $\mu$ is nearly silver, then it cannot be extended to a P-measure in  $V^{\SSi_\omega}$.
\end{cor}
\begin{proof} First note that, in general, for a measure $\mu$ and $f\in\omega^\omega$ we have $\mu f^{-1}$ vanishing on points exactly when $f$ is $\mu$-null-to-one.

Assume $\mu$ can be extended to a P-measure $\mu'$ in $V^{\SSi_\omega}$. Let $g\in\omega^\omega\cap V$ be $\mu$-null-to-one. Then by Proposition \ref{P-hered} we get that $\mu' g^{-1}$ is a P-measure. Since $\mu' g^{-1}$ extends $\mu
	g^{-1}$, from Theorem \ref{Nearly-no} we get that $\mu g^{-1}$ cannot be silver. 

\end{proof}

\begin{cor} Suppose that $\mu$ is nearly Dirac. Then $\mu$ cannot be extended to a P-measure in the $\omega_2$ iteration of $\SSi_\omega$. In particular, no ultrafilter and no ultrafilter extension of the asymptotic density can be extended to a P-measure in this model.
\end{cor}

\begin{prob} Assume $\mathsf{CH}$. Is there a measure that is not nearly silver? 
\end{prob}

If the answer is negative, then Theorem \ref{nonuni} and Proposition \ref{extending} imply that there are no P-measures in the model obtained by $\omega_2$ iteration of $\SSi_\omega$.
Note that the answer to the question: ``Is there a measure which is not nearly Dirac (under $\mathsf{CH}$)?'' is positive, see Section \ref{themeasure}.

Recall that a measure $\mu$ carried by a Boolean algebra $\mathbb{A}$ induces a pseudo-metric on $\mathbb{A}$: $d_\mu(A,B) = \mu(A\triangle B)$. The \emph{Maharam type} of $\mu$ is the density of the pseudo-metric space $(\mathbb{A}, d_\mu)$. In
general, the Maharam type of a measure $\mu$ on a space $K$ is the density of $\mathrm{Bor}(K)_{/\mu = 0}$ endowed with the pseudometric $d_\mu$. Notice that if there is $c>0$ and a family $(B_\alpha)_{\alpha<\kappa}$ of Borel sets such that $\mu(B_\alpha
\triangle B_\beta) > c$ for each $\alpha \ne \beta < \kappa$, then $\mu$ has the Maharam type at least $\kappa$.

For a  strictly increasing function $g\colon \omega\to \omega$ let $\widehat{g}\colon \omega\to \omega$ be the finite-to-one function defined by 
\[ \widehat{g}(k) = n \iff g(n) \leq k < g(n+1). \]

\begin{thm}\label{Maharam} Suppose that $\mu$ is a measure on $\omega$ which is not nearly silver. Then $\mu$ has an uncountable Maharam type.
\end{thm}

\begin{proof}
	We will construct sequences $(f_\alpha)_{\alpha<\omega_1}$ and $(g_\alpha)_{\alpha<\omega_1}$ of strictly increasing functions in $\omega^\omega$, a sequence $(r_\alpha)_{\alpha<\omega_1}$ of positive reals and a sequence
	$(N_\alpha)_{\alpha<\omega_1}$ of subsets of $\omega$. First, we will show how to obtain $f_\alpha, N_\alpha, r_\alpha$ having $g_\alpha$ defined and then how to obtain $g_\alpha$ having defined $f_\beta$ for all $\beta < \alpha$.

Suppose that $\alpha<\omega_1$ and that $g_\alpha$ is defined. Then let $r_\alpha, N_\alpha, f_\alpha$ witness that the measure $\mu\widehat{g}_\alpha^{-1}$ is not silver, i.e that
	for every $\epsilon \in \{-1,1\}^\omega$ 
	\[ \mu\big(\widehat{g}_\alpha^{-1}\big[ \bigcup_n [f_\alpha(n), f_\alpha(n+1)) \cap N^{\epsilon(n)}_\alpha\big] \big) > r_\alpha. \]

	Let $g_0$ be the identity function. For $\alpha<\omega_1$ let $g_{\alpha+1} = g_\alpha \circ f_\alpha$. If $\xi<\omega_1$ is a limit ordinal let $g_\xi$ be the function such that the interval partition induced by $g_\xi$ dominates all the interval partitions induced
	by $g_\alpha$ for $\alpha<\xi$ and $g_\xi[\omega]$ is a pseudo-intersection of $g_\alpha[\omega]$ for $\alpha<\xi$.

	Passing to a subsequence, if needed, we may assume that there is $r$ such that $r=r_\alpha$ for each $\alpha<\omega_1$.

	Now, for $\alpha<\omega_1$ and $\epsilon\in \{-1,1\}^\omega$ let 
	\[ A^\alpha_\epsilon = \widehat{g}_\alpha^{-1}\big[ \bigcup_n [f_\alpha(n), f_\alpha(n+1)) \cap N_\alpha^{\epsilon(n)}\big]. \]
	Then
	\[ \mu((A^\alpha_\epsilon))> r \]
	for each $\alpha<\omega_1$ and $\epsilon\in \{-1,1\}^\omega$.

	Notice that whenever $\alpha < \beta$ and $\epsilon, \epsilon' \in \{-1,1\}^\omega$ there is $\epsilon'' \in \{-1,1\}^\omega$ such that 
	\[ A^\alpha_\epsilon \triangle A^\beta_{\epsilon'} =^* A^\alpha_{\epsilon''}. \]

	Denote $A^\alpha = A^\alpha_\epsilon$ if $\epsilon(n)=1$ for each $n$.
	Then for each $\alpha<\beta<\omega_1$ we have
	\[ \mu(A^\alpha \triangle A^\beta) > r \]
	and therefore the sequence $(A^\alpha)$ proves that $\mu$ has uncountable Maharam type.
\end{proof}

Notice that P-points can be treated as P-measures of Maharam type 2 and so Chodounsk\'y-Guzm\'an theorem can be stated as: there are no P-measures of Maharam type 2 in the Silver model. Using the above theorem we may prove the
following: 

\begin{cor} There are no P-measures of countable Maharam type in the model obtained by $\omega_2$ iteration of $\SSi_\omega$.
\end{cor}

\begin{proof} This is a consequence of Theorem \ref{Maharam}, Theorem \ref{Nearly-no} and Proposition \ref{extending}.
\end{proof}

In Section \ref{themeasure} we will show that there exist non-atomic P-measures of countable Maharam type (under $\mathsf{CH}$).

\section{Rapid Filters and P-measures}
Let $\F$ be some filter containing all co-finite sets.  We will say that $\F$ is {\it rapid} if for any $g\in\omega^\omega$ there is an increasing $f\in\omega^\omega$ such that $f>g$ and $f[\omega]\in \F$.

\begin{thm} \label{rapid}
If $\F$ is a rapid filter then $\SSi_\omega$ forces that there is no P-measure extending $\F$.
\label{nopmeasure}
\end{thm}
\begin{proof}
Suppose that $\mu$ is a measure extending $\F$ in $V[G]$. Let $D_i(s)$ be as above.

	There are two cases: either for every $\eps>0$ and every $n$ there is $m>n$ such that  $ \mu \left(\bigcap_{n\leqslant k<m} D_0(s_k)\right)<\eps$ or there are $\eps>0$, $n$ such that for any $m>n$ we have $\mu\left(\bigcap_{n\leqslant k<m}D_0(s_k)\right)>\eps$. 
\bigskip

	\textbf{Case 1.}\;\; Using the $\omega^\omega$-bounding property we can choose an increasing function $h$ from $V$ so that $n\in \omega$ \[ \mu\left(\bigcap_{k\in[h(n),h(n+1))} D_0(s_k)\right)<\frac{1}{2^{n+2}}.\]  We put \[ M_n=\bigcap_{k\in[h(n),h(n+1))}D_0(s_k).\]

Let $p=\langle p_0,p_1,...\rangle $ be any condition forcing all the above and such that 
$$
 p\Vdash    \forall n \;\dot Z\cup g(n) \supseteq \dot M_n  
$$
	for some $\dot Z, g$ (again $g\in V$ by $\omega^\omega$-bounding and perhaps extending $p$). We want to show that $\mu (Z)=1$. 

We can assume (by potentially increasing $g$) that for each $n$ we have
\[
 \forall k\in[h(n),h(n+1)) \quad g(n)>\min \Cod (s_k).
\]
Now let $f>g$ be an increasing function such that $f[\omega]\in\F$.
We have $p\forc\dot Z \supseteq \dot M_n \setminus f(n)$.

Let $q=\Anp{q_0,q_1,...} \leqslant p$ be a condition such that for any $n$ $$\forall k\in [h(n),h(n+1)) \; q\forc f(n)\in \dot D_0(\dot s_k).$$
This can be done because $f(n)>\min \Cod(p_k)$. 

Now $q$ forces that $f(n)\in \dot M_n$ and so $f(n)\in \dot Z$.
We can conclude that $\dot Z \in \F$ and so 
\[ q\forc \dot\mu\dot Z =1 \;\land \; \sum_n\dot\mu \dot M_n < \frac{1}{2}\]
where $\dot Z$ was a name for an arbitrary pseudounion of $M_n$.

	\textbf{Case 2.} Choose $N$ and $\eps$ such that for every $m>N$ we have
$$\mu\left(\bigcap_{k\in[N,m)}D_0(s_k)\right)>\eps.$$

Take any $p=\langle p_0,p_1,...\rangle $ forcing all the above and such that (for some $\dot Z,g$)
$$
    p \Vdash \forall k\geqslant N\; \dot Z\setminus g(k) \subseteq \dot D_0(\dot s_{k}).  
$$
Again assume that $g(k)>\min\Cod(p_{k})$ and let $f>g$ such that $f[\omega\setminus N]\in \F$.

Now we can choose $q \leqslant p$ such that for any $k\geqslant N$ 
$$
q \forc f(k)\not \in \dot D_0(\dot s_k).
$$

	Now we can see that $q \forc \dot\mu\dot (Z) = 0 \;\land\; \lim_m \dot\mu \left(\bigcup_{k\in[N,m)}\dot D_0(\dot s_{k}) \right) \geqslant \eps.$ 
\end{proof}
Finally, we can arrive at the following corollary.
\begin{cor}
    If $\F$ is a rapid filter in $V$, $V[G]$ is a model obtained by forcing with $\SSi_\omega$ and $W\supseteq V[G]$ is a model with the $\omega^\omega$-bounding property over $V[G]$ then $\F$ cannot be extended to a P-measure in $W$. 
\end{cor}

\section{More on P-measures under CH} \label{themeasure}

In this section, we enclose several remarks on measures on $\omega$.

\begin{exa}\label{Random-in-omega} Let $\mathbb{B}$ be the measure algebra (i.e. $\mathbb{B} = \mathrm{Bor}(2^\omega)_{/\lambda = 0}$) and let $K$ be its Stone space. Then, by Sikorski extension theorem (see e.g. \cite[Lemma 1.1]{Kunen}), $K$ can be embedded homeomorphically into
	$\omega^*$. Of course, $K$ supports the Lebesgue measure, and this measure induces naturally a measure $\lambda$ on $\omega$. 
\end{exa}

What is interesting about the measure $\lambda$ from Example \ref{Random-in-omega} is that it is combinatorially far from the density measures, in particular from the atomless measures described in Section \ref{measures-on-omega}. The reason lies in the Maharam type of this measure.  

Namely, the measure from Example \ref{Random-in-omega} has the same Maharam type as the Lebesgue measure, i.e. countable, whereas density measures have uncountable Maharam types. The latter fact follows from \cite[1J]{FremlinTalagrand}, but we will give a more direct argument here, due to Grzegorz Plebanek. 
\begin{thm} Every density measure $\mu$ has the Maharam type $\mathfrak{c}$.
\end{thm}

\begin{proof}
	Let $\lambda_{\mathfrak{c}}$ be the standard Haar measure on $2^\mathfrak{c}$. By \cite[491G]{Fremlin-MT4} $\lambda_\mathfrak{c}$ has a uniformly distributed sequence $(x_n)$, which in particular means that the function $\varphi\colon \mathrm{Clop}(2^\mathfrak{c}) \to \mathcal{P}(\omega)$ defined
	by $\varphi(C) = \{n\colon x_n \in C\}$ is a Boolean homomorphism such that $d(\varphi(A))=\lambda_\mathfrak{c}(A)$ for each $A\in \mathrm{Clop}(2^\mathfrak{c})$. For $\alpha<\mathfrak{c}$ let $C_\alpha = \{x\in 2^\mathfrak{c}\colon x(\alpha)=1\}$ and let
	$A_\alpha = \varphi(C_\alpha)$. Then $(A_\alpha)$ forms a family such that $d(A_\alpha)=1/2$ and $d(A_\alpha \triangle A_\beta) = 1/4$ for $\alpha, \beta< \mathfrak{c}$, $\alpha\ne \beta$. If $\mu$ extends $d$, then
	$\mu(A_\alpha \triangle A_\beta)>1/4$ for each $\alpha<\mathfrak{c}$  So, every measure extending the asymptotic density has Maharam type $\mathfrak{c}$.
\end{proof}

The above theorem says, informally, that the asymptotic density has 'Maharam type` $\mathfrak{c}$. It follows that the measure $\lambda$ from Example \ref{Random-in-omega} does not extend the asymptotic density.

We will show that under $CH$ we can construct measures as in Example \ref{Random-in-omega} with some additional properties. We start with a theorem by Kunen. 

\begin{thm}\label{Kuna} (\cite[Theorem 1.2]{Kunen}). 
	Assume $\mathsf{CH}$. Then the space $K = \mathrm{St}(\mathbb{B})$ can be embedded into $\omega^*$ as a P-set. In other words, there is a surjective Boolean homomorphism $\varphi\colon \mathcal{P}(\omega)/fin \to \mathbb{B}$ whose kernel is a P-ideal. 
\end{thm}

Now we will show that if $\mathbb{B}$ is embedded as above, then the induced measure $\lambda$ is a P-measure.

\begin{prop} Assume that $\mathbb{B}$ is a Boolean algebra supporting a $\sigma$-additive measure $\nu$. Suppose that $\varphi\colon \mathcal{P}(\omega)/fin \to \mathbb{B}$ is a surjective Boolean homomorphism such that $\mathrm{ker}(\varphi)$ is a P-ideal. Let $\mu$ be a measure on $\omega$ defined by $\mu(A) =
	\nu(\varphi(A))$. Then $\mu$ is a P-measure.
\end{prop}
\begin{proof}  Let $\mathcal{F} = \{N\subseteq \omega\colon \varphi(N)=1\}$. Notice that $\mathcal{F}$ is a P-filter and that $\mathcal{F} = \{N\subseteq \omega\colon  \mu(N)=1\}$.
	Let $(A_n)$ be a decreasing sequence of subsets of $\omega$ and let $B_n = \varphi(A_n)$. Then $(B_n)$ is a decreasing sequence of elements of $\mathbb{B}$. There is $B\in \mathbb{B}$ such that $\lambda(B) = \lim_{n\to \infty}
	\lambda(B_n)$.

	Let $A\subseteq \omega$ be such that $\varphi(A) = B$. Then $\mu(A \setminus A_n)=0$ and so \[ A'_n =  (A \cap A_n) \cup A^c  \in \mathcal{F} \] for each $n$. As $\mathcal{F}$ is a P-filter, there is $N\in \mathcal{F}$ such that $N \subseteq^* A'_n$
	for every $n$. Then $N\cap A \subseteq^* A_n$ for every $n$ and $\mu(N\cap A) = \lim_{n\to \infty}\mu(A_n)$. As $(A_n)$ is arbitrary, $\mu$ is a P-measure.
\end{proof}

As a corollary of the above and of Theorem \ref{Kuna} we get the following. 

\begin{cor} Under $\mathsf{CH}$ there is a P-measure of countable Maharam type. 
\end{cor}

We will show that we can strengthen Kunen's construction to obtain a measure that is not nearly Dirac. To do it we have to take a closer look at the proof of Theorem \ref{Kuna}. Let $\mathcal{A} \subseteq P(\omega)$ be a family closed under finite
modifications. For a homomorphism $\psi\colon \mathcal{A} \to
\mathbb{B}$ define $\psi^*\colon P(\omega) \to \mathbb{B}$ and $\psi_*\colon P(\omega)\to \mathbb{B}$ by 
\[ \psi^*(X) = \bigwedge \{\psi(A)\colon A\in \mathbb{A}, X\subseteq A\} \]
and
\[ \psi_*(X) = \bigvee \{\psi(A)\colon A\in \mathbb{A}, A\subseteq X\}. \]
The main tool in the proof of Theorem \ref{Kuna} is the following ingredient of the proof of Sikorski's theorem. By $\mathbb{A}(X)$ we will denote the Boolean algebra generated by $\mathbb{A}$ and $X$.

\begin{lem}\label{Sikorski} Let $\mathbb{A} \subseteq P(\omega)$ be a countable Boolean algebra (closed under finite modifications) and let $\psi\colon \mathbb{A} \to \mathbb{B}$ be a homomorphism. If $X\subseteq \omega$ and $B\in \mathbb{B}$ is such that $\psi_*(X) \subseteq B \subseteq
	\psi^*(X)$, then $\psi$ can be extended to a homomorphism $\psi'\colon \mathbb{A}(X) \to \mathbb{B}$ in such a way that $\psi'(X) = B$.
\end{lem}

Theorem \ref{Kuna} can be proved by transfinite induction, by taking care of the following conditions:
\begin{itemize}
	\item[(1)] for every $B\in \mathbb{B}$ there is $X\subseteq \omega$ such that $\varphi(X) = B$.
	\item[(2)] for every countable $\mathcal{A} \subseteq \mathcal{P}(\omega)$, such that $\lambda(\varphi(A_\alpha))=1$ for each $\alpha<\beta$, there is a pseudo-intersection $X\subseteq \omega$ of $\mathcal{A}$ such that
		$\lambda(\varphi(X))=1$.
\end{itemize}

\begin{prop}\label{not-nearly-Dirac} Assume $\mathsf{CH}$. There is a P-measure that is not nearly Dirac.
\end{prop}
\begin{proof}
	We proceed as explained above, adding one more condition on $\varphi$.
	\begin{itemize}
		\item[(3)] for every finite-to-1 function $f\colon \omega\to \omega$ there is $N\subseteq \omega$ such that 
			\[ \lambda(\varphi(f^{-1}[N])) =\frac{1}{2}. \]
	\end{itemize}
	It is immediate (with a little use of \ref{rk-is-rb}) that (3) implies that the measure $\mu$ defined by $\mu(X) = \lambda(\varphi(X))$ is not nearly Dirac.

	To achieve (3) suppose that $\mathbb{A}$ is countable and a function $f\colon \omega\to \omega$ is finite-to-1. By induction we may find $N\subseteq \omega$ such that
	\[ | A \cap f^{-1}[N] | = | A \setminus f^{-1}[N] | = \omega \]
	for every infinite $A\in \mathbb{A}$.

	It means that $\varphi_*(A \cap f^{-1}[N]) = 0$ and $\varphi^*( A \cap f^{-1}[N] ) = 1$ and we can use Lemma \ref{Sikorski}.
\end{proof}

\section{A model with P-measures and without P-points}

In this section, we will prove the following theorem. Denote \[ \mathbb{B}_\kappa = \mathrm{Bor}(2^\kappa)/_{\lambda_\kappa = 0},\] where $\lambda_\kappa$ is the standard Haar measure on $2^\kappa$. Recall that $\mathbb{B}_\kappa$ is the complete Boolean
algebra adding $\kappa$ random reals.

\begin{thm}\label{mmain} It is relatively consistent with $\mathsf{ZFC}$ that there is a P-measure,  there is no P-point and $2^{\aleph_0}$ is arbitrarily large.
\end{thm}

The proof will rely on two results concerning the random forcing. The first one is proved in \cite[Theorem 4.15]{p_measures_random}.

\begin{thm}[Borodulin-Nadzieja, Sobota]\label{p_measures_random}
	If there is a P-point in the ground model, then after forcing with $\mathbb{B}_\kappa$, for any $\kappa$, there is a P-measure.
\end{thm}

The following is the main ingredient of the proof that there are no selective ultrafilters in the random model, see \cite[Theorem 5.1]{Kunen} and  \cite[Corollary 5.7]{p_measures_random} for an alternative proof.

\begin{thm}[Kunen]\label{selective}
No selective ultrafilter from the ground model can be extended to a P-point after forcing with $\mathbb{B}_\kappa$ for $\kappa \geq \omega$.
\end{thm}

Having these two theorems and Problem \ref{main_problem} in mind, our attention should be attracted by the following theorem, see \cite[Theorem 5.13 from Chapter VI, and Theorem 4.1 from Chapter XVIII]{pif}.

\begin{thm}[Shelah]\label{unique_Shelah} It is consistent that there is a P-point and all P-points are selective.
\end{thm}

Let $V$ be a model witnessing the situation as in the above theorem. If we force over $V$ with a random forcing, then by Theorem \ref{selective} no ultrafilter from $V$ can be extended to a P-point and, by Theorem \ref{p_measures_random} there is a
P-measure in the extension. It sounds like we are close to proving Theorem \ref{mmain}. However, a priori it is possible that by forcing with $\mathbb{B}$ we add some new P-points (we will come back to this issue in the remarks after Question
\ref{random_question}). So, to obtain a model in which there are P-measures but no
P-points, we have to work a little bit more. 

First, we need a strengthening of Theorem \ref{selective}. We will show that not only selective ultrafilters cannot be extended to P-points (in the random extensions) but in fact all ultrafilters that are nearly coherent with selective ultrafilters. 

\begin{df}[A. Blass, see \cite{ncf1}]
	Let $\mathcal{F}_0$ and $\mathcal{F}_1$ be filters on $\omega$. We say that $\mathcal{F}_0$ and $\mathcal{F}_1$ are nearly coherent if there is a finite to one function $f\colon \omega\to\omega$ such that $f[\mathcal{F}_0] \cup f[\mathcal{F}_1]$ generates a filter.
\end{df}

Note that if $\mathcal{U}$ is an ultrafilter, then $\mathcal{F}$, $\mathcal{U}$ are nearly coherent if and only if there is a finite to one function $f\colon \omega\to\omega$ such that $f[\mathcal{F}] \subseteq f[\mathcal{U}]$.

For a $\mathbb{B}_\kappa$-name for an ultrafilter $\dot{\mathcal{V}}$ and $p\in \mathbb{B}_\kappa$ we denote
\[ \dot{\mathcal{V}}[p] = \{A\subseteq \omega\colon p \Vdash A \in \dot{\mathcal{V}} \}. \]

\begin{thm}\label{not_coherent}
Let $\mathcal{U}$ be a selective ultrafilter and $\dot{\mathcal{V}}$ be a $\mathbb{B}_\kappa$-name such that $\mathbb{B}_\kappa \Vdash\dot{\mathcal{V}}\text{ is a P-point}$. Then for all $p\in\mathbb{B}_\kappa$, the filters $\mathcal{U}$ and $\dot{\mathcal{V}}[p]$ are not nearly coherent.
\end{thm}

\begin{proof} 
First note that if $\mathcal{F}$ is a filter which is Rudin-Blass above $\mathcal{U}$, which is witnessed by $f\colon \omega \to \omega$, then $\mathcal{F}$ can not be extended to a P-point after forcing with $\mathbb{B}_\kappa$. Indeed, suppose that 
	$\dot{\mathcal{G}}$ is a P-point extending $\mathcal{F}$. Then $\mathcal{U}\subseteq f[\dot{\mathcal{G}}]$, in other words $\mathcal{U}$ can be extended to $f[\dot{\mathcal{G}}]$ which is a P-point. This is impossible by Theorem \ref{selective}.

Assume now that there is $p_0\in\mathbb{B}_\kappa$ such that $\mathcal{U}$ and $\dot{\mathcal{V}}[p_0]$ are nearly coherent, and that this is witnessed by $h\colon \omega\to\omega$. It means that $h[\dot{\mathcal{V}}[p_0]]\subseteq h[\mathcal{U}]$. By the
	previous paragraph and the fact that $h[\mathcal{U}]$ is also selective, we have that $h[\mathcal{U}]\nleq_{RB}\dot{\mathcal{V}}[p]$ for every $p\in\mathbb{B}_\kappa$. 
So, for every $p\leq p_0$, there is $A_p\in h[\mathcal{U}]$ such that $h^{-1}[A_p]\notin \dot{\mathcal{V}}[p]$, which implies that there is $q\leq p$ such that $q\Vdash A_p\notin\dot{\mathcal{V}}$. 
	Therefore, the set \[ D=\{p\leq p_0\colon \exists A\in h[\mathcal{U}] \ p\Vdash h^{-1}[A]\notin \dot{\mathcal{V}} \}\] is dense below $p_0$.
	Let $\mathcal{A}\subseteq D$ be a maximal antichain below $p_0$. 
	As $\mathbb{B}_\kappa$ is ccc, $\mathcal{A}$ is necessarily countable, so there is $X\in h[\mathcal{U}]$ such that $X\subseteq^* A_p$, for all $p\in\mathcal{A}$. It follows that $p\Vdash h^{-1}[X]\notin\dot{\mathcal{V}}$ for all $p\in\mathcal{A}$
	(otherwise, for some $p\in\mathcal{A}$, there exist $q\leq p$ such that $q\Vdash h^{-1}[A_p]\in\dot{\mathcal{V}}$, which is a contradiction). Since $\mathcal{A}$ is a maximal antichain below $p_0$, it follows that $p_0\Vdash h^{-1}[X]\notin
	\dot{\mathcal{V}}$, which implies that $p_0\Vdash \omega\setminus h^{-1}[X]\in\mathcal{\dot{V}}$. Therefore, $\omega\setminus X\in h[\dot{\mathcal{V}}[p_0]]$. Thus, we have $X\in h[\mathcal{U}]$ and $\omega\setminus X\in
	h[\dot{\mathcal{V}}[p_0]]$, which contradicts our assumption that $h[\dot{\mathcal{V}}[p_0]]\subseteq h[\mathcal{U}]$.
\end{proof}

\begin{rem}
	Having Theorem \ref{not_coherent} and Theorem \ref{selective}, we may face the following temptation. Recall that the principle of Near Coherence of Filters (NCF) says that any two different filters are nearly coherent. It is consistent (e.g. it is implied
	by $\mathfrak{u}<\mathfrak{g}$, see \cite{Blass}). 
	So, if we have NCF with the
	existence of selective ultrafilter, then by Theorem \ref{not_coherent} no ultrafilter would be extendible to a P-point in the random extensions (and by Theorem \ref{p_measures_random} there are P-measures in the random extensions). So, by forcing with
	$\mathbb{B}_{\omega_2}$ over such a model, we would obtain a model with P-measures but without P-points. Unfortunately, this approach does not work: NCF implies that there are no selective ultrafilters. Indeed, suppose that  NCF holds and that $\mathcal{U}$ is a
	selective ultrafilter. NCF implies that $\mathfrak{u}<\mathfrak{d}$ (see \cite{coherentBlass}). So, there is an ultrafilter $\mathcal{V}$ generated by less than $\mathfrak{d}$ sets. By NCF we have $h[\mathcal{V}]\leq_{RB}
	h[\mathcal{U}]$.
	But $h[\mathcal{U}]$ is selective, and so $h[\mathcal{V}]$ is also selective which is impossible, see \cite[Remark 9.24]{Blass}. 	
\end{rem}

Now, we will prove Theorem \ref{unique_Shelah} in a slightly stronger form in a series of lemmas. We will extensively use filter games, both the classical one introduced by Laflamme in \cite{Laflamme} and its modification considered by Eisworth in
\cite{todd}.

\begin{df}[P-filter game]
	Let $\mathcal{F}$ be a  filter on $\omega$. Consider the game $\mathcal{G}(\mathcal{F})$ between Adam and Eve defined in the following way:
	\begin{itemize}
		\item At round $n$ Adam plays a set $A_n \in \mathcal{F}$ and Eve responds with a finite set $F_n \subseteq A_n$.
	\end{itemize}
	Eve wins if $\bigcup_n F_n \in \mathcal{F}$.
\end{df}

Using the above game gives us a handy characterization of the property of being a P-filter in the realm of non-meager filters, see \cite[Theorem 2.15]{Laflamme}

\begin{thm}[Laflamme]\label{laflamme}
	Suppose that $\mathcal{F}$ is non-meager. Then $\mathcal{F}$ is a P-filter if and only if Adam does not have a winning strategy in $\mathcal{G}(\mathcal{F})$.
\end{thm}

We will also need a game defined in the same spirit by Eisworth in \cite[Definition 3.3]{todd}.

\begin{df}\label{Eisworth-game}
Let $\mathcal{F}_0,\mathcal{F}_1$ be filters on $\omega$. The game $\mathcal{G}(\mathcal{F}_0,\mathcal{F}_1)$ between Adam and Eve is defined as follows: a complete stage of the game consists of two rounds:
\begin{itemize}
\item At round $2n$, Adam plays sets $A_n\in\mathcal{F}_0$, and Eve answers with $a_n\in A_n$.
\item At round $2n+1$, Adam plays sets $B_n\in\mathcal{F}_1$, and Eve answers with a non-empty $F_n\in[B_n]^{<\omega}$.
\end{itemize}
Eve wins  if $\{a_n\colon n\in\omega\}\in \mathcal{F}_0$ and $\bigcup_n F_n \in\mathcal{F}_1$.
\end{df}

For technical reasons, we will ask Eve to play different numbers $a_n$ and pairwise disjoint sets $F_n$. Taking into account her aim in the game, she should not complain about this additional rule.

Our first step is to prove that under suitable assumptions, Adam has no winning strategy in the previous game. We need the following lemma. The reader can find similar theorems in \cite[Corollary 2.3]{todd} and \cite{David-thesis}.

\begin{lem}[Eisworth]\label{todd_lemma}
Let $\mathcal{F}_0,\mathcal{F}_1$ be filters which are not nearly coherent. Let $\mathcal{I}=\{I_n\colon n\in\omega\}$ be an interval partition of $\omega$. Then there is an interval partition $\mathcal{J}=\{J_n\colon n\in\omega\}$ such that each
	 $J_n$ is a union of elements of $\mathcal{I}$, and
\begin{equation*}
    \bigcup_{k\in\omega} J_{4k+1}\in\mathcal{F}_0\text{   and   }    \bigcup_{k\in\omega} J_{4k+3}\in\mathcal{F}_1
\end{equation*}
\end{lem}

The following lemma is essentially \cite[Theorem 1]{todd}.

\begin{lem}
Let $\mathcal{U}$ be a selective ultrafilter, let $\mathcal{F}$ be a P-filter, and suppose that they are not nearly coherent. Then Adam has no winning strategy in the game $\mathcal{G}(\mathcal{U},\mathcal{F})$.
\end{lem}

\begin{proof}
	Let $\Sigma$ be a strategy for Adam. We will produce a run of the game in which he uses $\Sigma$ but nevertheless, Eve wins the game. We think of $\Sigma$ in a natural way as a tree such that each $\sigma \in \Sigma$ is a sequence of the
	form $(A_0, a_0, B_0, F_0, A_1, \dots, A_n)$ or $(A_0, a_0, B_0, F_0, A_1, \dots, B_n)$, where $A_i$, $a_i$, $B_i$, $F_i$ are as described in Definition \ref{Eisworth-game}. 

	Note that $\Sigma$ is a countable tree, so there is $A_\omega\in\mathcal{U}$ which is a pseudointersection of all the elements of $\mathcal{U}$ appearing in $\Sigma$ and, similarly, there is $B_\omega\in \mathcal{F}$ which is a pseudointersection
	of all the elements of $\mathcal{F}$ which appeared in $\Sigma$. We can assume that $A_\omega \subseteq A_0$, the first move of Adam.

	
Fix $n\in \omega$. Notice that there are finitely many $\sigma\in\Sigma$ such that every Eve's move in $\sigma$ is a subset of $n$; we say that such $\sigma$ has rank $\leq n$. Denote by $\Sigma_n$ all the elements of $\Sigma$ having rank $\leq n$.
	Furthermore, let $E_n$ denote the collection of sets $A\in\mathcal{U}$ such that there is a $\sigma\in\Sigma_n$ in which the set $A$ appears. Similarly, denote by $O_n$ the collection of sets $B\in\mathcal{F}$ such that there is
	$\sigma\in\Sigma_n$ in which the set $B$ appears. Notice that both $E_n$ and $O_n$ are finite. 
	
Now we define a sequence of natural numbers $\{k_n\colon n\in\omega\}$ as follows:
\begin{enumerate}
\item $k_0=0$ and $k_1$ is such that $A_0 \cap k_1\neq\emptyset$.
\item Assume $k_i$ has been defined. Then choose $k_{i+1}$ such that $A_\omega\cap [k_i, k_{i+1}) \ne \emptyset$, $B_\omega\cap [k_i, k_{i+1}) \ne \emptyset$ and for all $A\in E_{k_i}$, $A_\omega\setminus k_{i+1}\subseteq A$, and for all $B\in O_{k_i}$, $ B_\omega\setminus k_{i+1}\subseteq B$.
\end{enumerate}

Having defined the sequence, consider the interval partition given by $I_n=[k_n,k_{n+1})$. Apply Lemma \ref{todd_lemma} to get a partition $\{J_n\colon n\in\omega\}$ such that each $J_n$ is union of finitely many consecutive intervals $I_k$, and such that
\begin{equation*}
    \bigcup_{k\in\omega} J_{4k+1}\in\mathcal{U}\text{   and   }    \bigcup_{k\in\omega} J_{4k+3}\in\mathcal{F}.
\end{equation*}

	We may assume that $k_1<\min(J_1)$, and so $k_2\leq \min(J_1)$ (the minimum of each interval $J_n$ is of the form $k_l$). Since $\mathcal{U}$ is a selective ultrafilter, and $A_\omega \cap I_n \ne \emptyset$ for each $n$ we can find $A=\{a_n\colon n\in\omega\}\in\mathcal{U}$, $A\subseteq
	A_\omega$ such that
	$A\cap J_{4n+1}=\{a_n\}$ for all $n\in\omega$. Let $B=B_\omega\cap \bigcup_{n\in\omega}J_{4n+3}$. We claim that there is a run of the game in which Adam follows $\Sigma$ and Eve constructs the sets $A$ and $B$ along. Clearly, this suffices to
	prove that $\Sigma$ is not a winning strategy, as $A\in \mathcal{U}$ and $B\in \mathcal{F}$. Eve should play as follows:
\begin{enumerate}
	\item After Adam has played $A_0$, Eve plays $a_0$. Notice that it is a legal move (i.e. $a_0 \in A_0$) as $a_0 \in A_\omega $  and $A_\omega \subseteq A_0$.
\item Then Adam answers with $B_0$. Note that $\sigma=(A_0,a_0,B_0)$ has rank $\leq \max(J_1)+1=\min(J_{2})$, so $B_0\in O_{\min(J_{2})}$, which implies that $B_\omega\setminus \min(J_3)\subseteq B_0$, so $B_\omega\cap J_3\subseteq B_0$. Then Eve can
	play $F_0=B_\omega\cap J_3=B\cap J_3$ which is non-empty by the choice of the sequence $(k_n)$.
\item Let $A_1$ be the answer of Adam. The sequence $(A_0,a_0,B_0,F_0,A_1)$ has rank $\leq \min(J_4)$, so $A_1\in E_{\min(J_4)}$, which implies that $A_\omega\setminus \min(J_5)\subseteq A_1$, and thus, $a_1\in A_1$, which makes it a legal (and
	recommended) move for Eve. 
\item Eve continues playing $a_n$ and $F_n = B_\omega \cap J_{4n+3}$ accordingly.
\end{enumerate}
Note that $\bigcup_{n\in\omega}F_n=B$, and clearly $A=\{a_n\colon n\in\omega\}$.
\end{proof}

The following forcing notion was introduced by Shelah to prove Theorem \ref{unique_Shelah}.

\begin{df}
For each natural number $n$, let $S_n=\Pi_{k\leq n}2^{n}$ (note that here by $2^n$ we mean the set of sequences of bits having length $n$; to avoid misunderstandings in what follows we will denote the elements of $2^n$ by Roman letters $s, t, \dots$
	and the elements of $S_n$ by Greek letters $\sigma, \tau, ..$), and $S=\bigcup_{n\in\omega}S_n$. A tree $T\subseteq S$ is a family closed under initial segments. For a tree
	$T\subseteq S$, $\tau \in T$, and $n\in\omega$, define 
	\begin{itemize}
		\item $T^{[n]}=\{\sigma\in T\colon \vert \sigma\vert=n\}$,
		\item $succ_T(\sigma) = \{s\in 2^{|\sigma|} \colon \sigma^\frown s\in T\}$,
		\item $T\upharpoonright \tau=\{\sigma \in T\colon \tau \subseteq \sigma\vee \sigma \subseteq \tau\}$.
	\end{itemize}
Let $\mathcal{F}$ be a filter. For a tree $T\subseteq S$ and $k\in\omega$, define
\begin{equation*}
spt_{k}(T)=\{n\in\omega\colon \forall \tau \in T^{[n]} \ \forall s\in 2^k \ \exists t\in succ_T(\tau) \ s\subseteq t\}.
\end{equation*}

The forcing $SP^*(\mathcal{F})$ is defined as the collection of all trees $T\subseteq S$ such that for all $k\in\omega$, $spt_{k}(T)\in\mathcal{F}$. The order is given by set inclusion.
\end{df}

Notice that if $T\in SP^*(\mathcal{F})$ and $\tau \in T$, then $T\upharpoonright \tau \in SP^*(\mathcal{F})$.

\begin{rem}\label{sp-generic} The conditions of $SP^*(\mathcal{F})$ are slightly complicated so maybe this is a good moment to see how the forcing actually 'works'. Let $G$ be $SP^*(\mathcal{F})$-generic, and define $g \in [S]$ to be such that $\bigcup\bigcap
	G = \{g\}$. 
We can see it as a sequence of reals: for each natural number define $\dot{r}_n\colon \omega\setminus(n+1)\to 2$ as \[ \dot{r}_n(k)=1 \mbox{ if and only if }\dot{g}(k)(n)=1. \] Actually, we will rather treat it as a sequence of subsets of $\omega$:
	\[ \dot{R}_n=\{k\in\omega\colon \dot{r}_n(k)=1\}.\] A simple density argument shows that for all $n\in\omega$ we have
	$\dot{R}_n,\omega\setminus \dot{R}_n\in\mathcal{F}^+$. The family $\{\dot{R}_n\colon n\in \omega\}$ will be used in a crucial step of the proof.
\end{rem}

\begin{prop}[Shelah, see \cite{pif}]\label{omega-omega}
If $\mathcal{F}$ is a non-meager P-filter, the forcing $SP^*(\mathcal{F})$ is proper and $\omega^\omega$-bounding.
\end{prop}

\begin{proof}
We only prove that $SP^*(\mathcal{F})$ is $\omega^\omega$-bounding, the proof for properness follows similar lines, just choosing an elementary submodel as usual and taking care that each step of the construction lives inside the submodel. So let $p\in
	SP^*(\mathcal{F})$ be a condition and let $\dot{f}$ be a name for a function from $\omega$ to $\omega$. We are going to find a sequence of finite sets $\{W_n\colon n\in\omega\}$ and $q\leq p$ such that  $q\Vdash \dot{f}(n)\in W_n$ for each
	$n\in\omega$. To do so, we give a strategy to Adam in the P-filter game $\mathcal{G}(\mathcal{F})$.

\begin{itemize}
\item Adam starts by extending $p$ to a condition $p_0\leq p$ which decides the value of $\dot{f}(0)$, and plays the set $spt_1(p_0)$. 
\item Eve answers with a finite set $F_0\subseteq spt_1(p_0)$. Then Adam chooses a natural number $n_0>\max(F_0)$. Define $W_0=\{l\}$, where $l$ is such that $p_0\Vdash \dot{f}(0)=l$.
\item At round $k$, suppose Eve has previously played a set $F_{k-1}$, Adam has chosen $n_{k-1}>\max(F_{k-1})$ and defined conditions $\{p_{i}\colon i< k\}$ and finite set $\{W_i\colon i< k\}$ such that $p_i\Vdash \dot{f}(i)\in W_i$. Now, for each
	node $\sigma \in p_{k-1}^{[n_{k-1}]}$, let $q_s\leq p_{k-1}\upharpoonright \sigma$ be a condition which decides the value of $\dot{f}(k)$, and let Adam play the set $\bigcap \{spt_{k+1}(q_\sigma)\colon {\sigma \in p_{k-1}^{[n_{k-1}]}}\}$.
\item Eve answers with a finite set $F_k$. Then Adam chooses a natural number $n_k>\max(F_k)$, defines $p_k=\bigcup \{q_\sigma \colon {\sigma \in p_{k-1}^{[n_{k-1}]}}\}$, and \[ W_k=\{l\in\omega\colon \exists \sigma\in p_{k-1}^{[n_{k-1}]} \ q_\sigma \Vdash \dot{f}(k)=l
	\}.\]
		Notice that $F_k\subseteq spt_k(p_k)$ and $p_k \Vdash \dot{f}(k) \in W_k$.
\end{itemize}

This is not a winning strategy for Adam, because of Theorem \ref{laflamme}. So, there is a run of the game in which Eve wins. Let $\{p_n\colon n\in\omega\}$ and $\{W_n\colon n\in\omega\}$ be the sequences constructed along such a run. Define
$p_\omega=\bigcap_{n\in\omega} p_n$. Notice that $p_\omega^{[i]} = p_n^{[i]}$ for every $i\leq n$. So, as $n_k > \max F_k$ for each $k$, we have that $\bigcup_{l\ge k} F_l\subseteq spt_{k}(p_\omega)$ for each $k$. Hence, $p_\omega$ is a condition.
	Also, since $p_\omega \leq p_k$ for
each $k$, we have $p_\omega \Vdash \dot{f}(k)\in W_k$ and we are done.
	 
\end{proof}

The following is essentially a part of the proof of Shelah's theorem on the consistency of the existence of only one selective ultrafilter up to permutation, we include the proof for completeness.

\begin{lem}[S. Shelah, see \cite{pif}]\label{preservation_destruction}
Let $\mathcal{U}$ be a selective ultrafilter and $\mathcal{F}$ be a non-meager P-filter such that $\mathcal{U}$ and $\mathcal{F}$ are not nearly coherent. Let $\dot{\mathbb{Q}}$ be a $SP^*(\mathcal{F})$-name for a proper $\omega^\omega$-bounding forcing. Then:
\begin{enumerate}
\item $SP^*(\mathcal{F})$ preserves $\mathcal{U}$ as a selective ultrafilter.
\item $SP^*(\mathcal{F})\ast \dot{\mathbb{Q}}$ forces that $\mathcal{F}$ can not be extended to a P-point.
\end{enumerate}
\end{lem}

\begin{proof}
We first prove (1). Let $\dot{x}$ be a name for a subset of $\omega$ and let $p\in SP^*(\mathcal{F})$. 
	We are going to show that there is $q\leq p$ such that 
	\begin{enumerate}
		\item[(i)] $q \Vdash \dot{x}$ is disjoint with an element of $\mathcal{U}$ or 
		\item[(ii)] $q\Vdash \dot{x}$ has a subset belonging to $\mathcal{U}$. 
	\end{enumerate}
	It will mean that $\mathcal{U}$ generates an ultrafilter in the
	forcing extension, and we will be done.
	
	For $q\in SP^*(\mathcal{F})$ define $\dot{x}[q]=\{n\in\omega\colon \exists r\leq q \ r\Vdash n\in\dot{x}\}$. If there is a condition $q\leq p$ such that $\dot{x}[q]\notin \mathcal{U}$, then we are in the situation (i). So we may assume that for
	all $q\leq p$ $\dot{x}[q]\in\mathcal{U}$. We give a strategy for Adam in the game $\mathcal{G}(\mathcal{U},\mathcal{F})$. On the side, Adam will construct a decreasing sequence of conditions $\{p_n\colon n\in\omega\}$ and a sequence of natural
	numbers $\{a_n\colon n\in\omega\} = A \in \mathcal{U}$, and after a run of the game in which Eve wins, there will be a lower bound $p_\omega$ of $\{p_n\colon n\in \omega\}$ and $p_\omega \Vdash A \subseteq \dot{x}$.

We dictate to Adam the following strategy.
\begin{itemize}
\item Adam starts by playing $\dot{x}[p]$. Then Eve chooses $a_0\in \dot{x}[p]$. Now Adam extends $p$ to a condition $p_0 \Vdash a_0\in\dot{x}$, and plays the set $spt_1(p_0)$. Then Eve answers with a finite set $F_0\subseteq spt_1(p_0)$ and Adam
	chooses a natural number $n_0>\max(F_0)$.
\item At stage $k$, after Eve has played $F_{k-1}$, Adam has chosen $n_{k-1}$ and has defined $p_{k-1}$, we proceed as follows. Adam plays the set 
\begin{equation*}
A_k=\bigcap\{\dot{x}[p_{k-1}\upharpoonright \sigma]\colon \sigma \in p_{k-1}^{[n_{k-1}]}\}.
\end{equation*}

Then Eve answers with $a_k\in A_k$. For each $\sigma \in p_{k-1}^{[n_{k-1}]}$, Adam extends the condition $p_{k-1}\upharpoonright \sigma$ to a condition $q_\sigma$ which forces $a_k\in\dot{x}$. Then he plays the set
\begin{equation*}
    B_k=\bigcap\{spt_{k+1}(q_\sigma)\colon \sigma \in p_{k-1}^{[n_{k-1}]}\}
\end{equation*}
Then Eve answers with a finite set $F_k\subseteq B_k$, and Adam chooses $n_k>\max(F_k)$. Finally, he defines $p_k=\bigcup\{q_\sigma \colon \sigma \in p_{k-1}^{[n_{k-1}]}\}$. Note that $p_k \Vdash a_k\in\dot{x}$.
\end{itemize}

	By Theorem \ref{laflamme} the above strategy is not winning. Let $A = \{a_k\colon k\in\omega\}$, $\{F_k\colon k\in\omega\}$ and $\{p_k\colon k\in\omega\}$ be the sequences constructed along a run of the game in which Eve wins. Define
	$p_\omega=\bigcap_{k\in\omega}p_n$. Similarly as in the proof of Proposition \ref{omega-omega} we have that $p_\omega \in SP^*(\mathcal{F})$. We also have that $p_\omega \Vdash A \subseteq \dot{x}$, and $A\in \mathcal{U}$. So, we are in the situation
	(ii), for $q=p_\omega$. This finishes the proof of (1).
\medskip

	We now prove (2). We will use the terminology of Remark \ref{sp-generic}. Let $\dot{\mathcal{V}}$ be a $SP^*(\mathcal{F})\ast \dot{\mathbb{Q}}$-name for an ultrafilter extending $\mathcal{F}$, and define $\dot{h}\colon \omega\to 2$ as
	$\dot{h}(n)=1$ if and only if $\dot{R}_n\in\dot{\mathcal{V}}$. So, \[ \dot{R}_n^{\dot{h}(n)}\in\dot{\mathcal{V}} \]
	for every $n$.

Fix a condition $(p,\dot{q})\in SP^*(\mathcal{F})\ast \dot{\mathbb{Q}}$. We have two cases: (a) there is  $(p_0,\dot{q}_0)\leq (p_0,\dot{q_0})$ which forces $\dot{h}$ to be eventually constant; (b) $(p,\dot{q})$ forces that $\dot{h}$ is not eventually constant. We only prove (b) since the argument for (a) is similar and slightly simpler. 

Since $\dot{\mathbb{Q}}$ is forced to be $\omega^\omega$-bounding, also $SP^*(\mathcal{F})\ast \dot{\mathbb{Q}}$ is $\omega^\omega$-bounding, and so there is a condition $(p_0,\dot{r}_0)\leq (p,\dot{q})$ and a function $f\in\omega^\omega$ such that
	for all $n\in\omega$, $n+1<f(n)$ and \[ (p_0,\dot{r}_0)\Vdash \exists i,j\in [f(n),f(n+1)) \ \dot{h}(i)\neq \dot{h}(j).\] 
	For each $n\in\omega$ define \[ \dot{W}_n=\bigcap_{j\in [f(n),f(n+1))}\dot{R}_j^{\dot{h}(j)}.\] Clearly, $\dot{W}_n$ is forced to be in $\dot{\mathcal{V}}$. We will show that $\{\dot{W}_n\colon n\in\omega\}$ has no pseudo-intersection in
	$\dot{\mathcal{V}}$. Indeed, let $\dot{Z}$ be a name for a pseudointersection of $\{\dot{W}_n\colon n\in\omega\}$. We may assume that there is $g
\in\omega^\omega$ in the ground model such that $(p_0,\dot{r}_0)$ forces $\dot{Z}\setminus g(n)\subseteq\dot{W}_n$ for each $n$ (extending and renaming $(p_0,\dot{r}_0)$ if needed). 
	
Again, we will play the P-filter game. We suggest the following strategy to Adam. Let $q_0=p_0$ and let $k_0>g(0)$. 
	At stage $n$, after defining $q_n$ and $k_n > g(n)$, Adam plays the set $A_n=spt_{f(n+1)}(q_n)\setminus k_n$. Eve responds with a finite set $F_n=\{a_0^n<\ldots<a_{m_n}^n\}\subseteq A_n$. We construct a sequence of decreasing conditions
	$\{p_j^n\colon j\leq m_n\}$ as follows. For each $\sigma \in q_n^{[a_0^n]}$ and $s \in 2^{n+1}$ using the fact that $a_0^n\in spt_{f(n+1)}(q_n)$ we can find $r_s^\sigma \in succ_{q_n}(\sigma)$ such that 
	\begin{itemize}
		\item $s \subseteq r_{s}^\sigma$, and
		\item for all $l\in [f(n),f(n+1))$ we have $r_s^\sigma(l)=0$.
	\end{itemize}
	Define \[ p_0^n=\bigcup\{q_n\upharpoonright (\sigma ^\frown r_s^\sigma) \colon \sigma \in q_n^{[a_0^n]}\land s \in 2^{n+1}\}. \]
	Notice that for each $j\in [f(n), f(n+1))$
	\[ (p_0^n, \dot{r}_0) \Vdash a^n_0 \notin \dot{R}_j \]
	and since there is $i\in [f(n), f(n+1))$ such that $(p_0, \dot{r}_0)\Vdash \dot{h}(i)=0$ we have also
	\[ (p_0^n, \dot{r}_0) \Vdash a^n_0 \notin \dot{W}_n. \]
Now, having defined $p_{j}^n$, for each $\sigma \in (p_{j}^n)^{[a_{j+1}^n]}$ and $s\in 2^{n+1}$, let $r_s^\sigma \in succ_{p^n_j}(\sigma)$ be such that 
\begin{itemize}
	\item $s \subseteq r_s^\sigma$, and 
	\item for all $l\in[f(n),f(n+1))$, $r_s^\sigma(l)=0$.
\end{itemize}
Define \[ p_{j+1}^n=\bigcup\{p_j^n\upharpoonright (\sigma ^\frown r_s^\sigma)\colon \sigma \in (p^n_j)^{[a_{j+1}^n]}\land s \in 2^{n+1}\}.\] 
Finally, let $q_{n+1}=p_{m_n}^n$, choose $k_{n+1}>\max\{g(n+1),\max(F_n)\}$, and notice that 
	\[ (q_{n+1}, \dot{r}_0) \Vdash F_n \cap \dot{W}_n = \emptyset. \]

According to Theorem \ref{laflamme} there is a run of the game in which the above strategy fails. Let $\{q_n\colon n\in\omega\}$ and $\{F_n\colon n\in\omega\}$ be the sequences constructed along such a run. Then $X=\bigcup_{n\in\omega}
F_n\in\mathcal{F}$. Define $q_\omega=\bigcap_{n\in\omega}q_n$ and notice that, as in the proof of Theorem \ref{omega-omega}, $q_\omega$ is a condition.  Then for each $n$
\[ (q_\omega,\dot{r}_0)\Vdash F_n \cap \dot{W}_n=\emptyset. \] 
But, as $\min F_n > g(n)$ and $(p_0, \dot{r}_0) \Vdash \dot{Z} \setminus g(n) \subseteq \dot{W}_n$ it means that
\[ (q_\omega,\dot{r}_0)\Vdash F_n \cap \dot{Z}=\emptyset. \] 
So, $(q_\omega,\dot{r}_0)$ forces $\dot{Z}$ to be outside $\dot{\mathcal{V}}$.
\end{proof}

Before the proof of the main theorem, we need one more fact: if we force with a random forcing over a model $V$, then every $\mathbb{B}_\kappa$-name for a non-principal ultrafilter on $\omega$ contains a substantial portion of a ground model ultrafilter. 

\begin{prop}\label{ppmeasure} Suppose that $\dot{\mathcal{V}}$ is a $\mathbb{B}_\kappa$-name for an ultrafilter on $\omega$. Let $p\in \mathbb{B}_\kappa$. Then,  there is a measure $\mu$ on $\omega$, vanishing on points, such that 
	\[ 	\{F\subseteq \omega\colon \mu(F)=1\} \subseteq \dot{\mathcal{V}}[p] \ ( = \{A\subseteq \omega\colon p \Vdash A \in \dot{\mathcal{V}} \}). `\]
Additionally, if $p$ forces that $\dot{\mathcal{V}}$ is a P-point, then $\mu$ can be assumed to be a P-measure.
\end{prop}
\begin{proof}
	Let $\dot{\mathcal{V}}$ be a name for an ultrafilter. Let $\mu$ be a measure on $\omega$ defined by
	\[ \mu(A)  = \lambda (\llbracket A \in \dot{\mathcal{V}} \rrbracket \cap p)/\lambda(p). \]
	Clearly, if $\mu(A)=1$, then $p \Vdash A\in \dot{\mathcal{V}}$ and so $A\in \dot{\mathcal{V}}[p]$. 

  Now, suppose that $p\Vdash \dot{\mathcal{V}}$ is a P-point.
	Let $(A_n)$ be a decreasing sequence of subsets of $\omega$. Suppose, towards the contradiction, that $\lim_{n\to \infty} \mu(A_n) = a>0$ and there is no pseudointersection $A$ of $(A_n)$ such that $\mu(A)=a$. In fact, we may assume that there is
	no pseudointersection $A$ of $(A_n)$ such that $\mu(A)>0$ (see \cite[Lemma 3.3]{p_measures_random}). Let $p = \bigwedge
	\llbracket A_n \in \dot{\mathcal{U}} \rrbracket$. Then $p\ne 0$ and $p\Vdash \forall n \ A_n \in \dot{\mathcal{V}}$. So, there is $q\leq p$ such that \[ q \Vdash \exists \dot{B} \in \dot{\mathcal{U}} \ \forall n \ \dot{B} \subseteq^* A_n. \]
	By the fact that $\mathbb{B}_\kappa$ is $\omega^\omega$-bounding,  there is $r\leq q$ and $A\in V$, $p \Vdash \dot{B}\subseteq A$ such that
	\[ r \Vdash  \forall n \ A \subseteq^* A_n. \] Then $r\leq \llbracket A \in \dot{\mathcal{U}}\rrbracket$ and so $\mu(A)>0$, a contradiction.
\end{proof}

In the proof of the main theorem we will only need the first part of the above proposition, the part about P-measure will be needed for other purposes. Actually, we will use only the following corollary.

\begin{cor}\label{nonme} Suppose that $\dot{\mathcal{V}}$ is a $\mathbb{B}_\kappa$-name for an ultrafilter on $\omega$. Then $\dot{\mathcal{V}}[p]$ is non-meager for each $p\in \mathbb{B}_\kappa$.
\end{cor}
\begin{proof} Suppose that $\mathcal{F}$ is a filter containing a family $\{A\subseteq \omega\colon \mu(A)=1\}$ for some measure on $\omega$. According to Proposition \ref{Daria} the filter $\mathcal{F}$ is non-meager and so, thanks to Proposition
	\ref{ppmeasure}, we are	done.
\end{proof}

\begin{rem} Recall that a filter $\mathcal{F}$ is called ccc if $\mathcal{P}(\omega)/\mathcal{F}^*$ is ccc. In the proof of Proposition \ref{Daria} we actually show that every ccc filter is nonmeager. It is not difficult to prove that if $\mathbb{P}$ is a ccc forcing notion,
	and $\dot{\mathcal{U}}$ is a $\mathbb{P}$-name for an ultrafilter on $\omega$, then there is a ground model ccc filter $\mathcal{F}$ such that $\dot{\mathcal{U}}$ extends $\mathcal{F}$ (see e.g. \cite{shelppoint}). As
	$\mathbb{B}_\kappa$ is obviously ccc, we could use this fact to prove Corollary \ref{nonme} instead of Proposition \ref{ppmeasure}.
\end{rem}

Now we are ready to prove the main theorem of this section.

\begin{proof}(of Theorem \ref{mmain})
	Assume $V$ is a model of $\mathsf{ZFC+CH+\diamondsuit(S)}$, where $S\subseteq\omega_2$ is stationary. Let $(A_\alpha:\alpha\in S)$ be a $\diamondsuit(S)$-guessing sequence. Let $\kappa\ge\omega_2$ be an uncountable regular cardinal. Let
	$\mathcal{U}$ be a selective ultrafilter and $\mathcal{V}$ a P-point which is not Rudin-Blass above $\mathcal{U}$. Define a countable support iteration $(\mathbb{P}_\alpha,\dot{\mathbb{Q}}_\alpha\colon \alpha<\omega_2)$ as follows:
\begin{enumerate}
	\item $\mathbb{P}_0=SP^*(\mathcal{U})$.
\item If $\alpha\notin S$, define $\dot{\mathbb{Q}}_\alpha$ to be the trivial forcing. 
\item If $\alpha\in S$, and $A_\alpha$ codifies a $\mathbb{P}_\alpha$-name for a non-meager P-filter $\dot{\mathcal{F}}$ which is not nearly coherent with $\mathcal{U}$, define $\mathbb{P}_\alpha\Vdash\dot{\mathbb{Q}}_\alpha=SP^*(\dot{\mathcal{F}})$; otherwise, let $\dot{\mathbb{Q}}_\alpha$ be the trivial forcing.
\end{enumerate}
Let $\mathbb{P}_{\omega_2}$ the resulting forcing, and let $\dot{\mathbb{B}}_\kappa$ be a $\mathbb{P}_{\omega_2}$-name for the forcing adding $\kappa$ random reals. Our working forcing is $\mathbb{P}_{\omega_2}\ast \dot{\mathbb{B}}_\kappa$.

Let $G\ast H$ be a $\mathbb{P}_{\omega_2}\ast \dot{\mathbb{B}}_\kappa$-generic filter, and for each $\alpha<\omega_2$ let $G_\alpha$ be the restriction of $G$ to $\mathbb{P}_\alpha$. Since all steps of the iteration preserve $\mathcal{U}$ (by Lemma
	\ref{preservation_destruction}), we have that  $\mathcal{U}$ continues to generate a selective ultrafilter in $V[G]$. So, by Theorem \ref{p_measures_random}, there is a P-measure in $V[G\ast H]$.

Now, suppose for the sake of contradiction, that there is a P-point $\mathcal{F}$ in $V[G*H]$. 
	By Lemma \ref{not_coherent} we have that in $V[G]$ the P-filter $\dot{\mathcal{F}}[1]$ is not nearly coherent with $\mathcal{U}$ and by Corollary \ref{nonme} it is non-meager. Let $\dot{\mathcal{H}}$ be a $\mathbb{P}_{\omega_2}$-name for the filter $\dot{\mathcal{F}}[1]$. For each
	$\alpha\in\omega_2$, define $\dot{\mathcal{H}}_\alpha=\dot{\mathcal{H}}\upharpoonright \mathbb{P}_\alpha$. Let 
	\[ C_0=\{\alpha\in\omega_2\colon \mathbb{P}_\alpha\Vdash\dot{\mathcal{H}}_\alpha=\dot{\mathcal{H}}\cap V[G_\alpha]\} \] and let
	\[ C_1=\{\alpha\in\omega_2:\mathbb{P}_\alpha \Vdash \text{ $\dot{\mathcal{H}}_\alpha\in V[G_\alpha]$ is not nearly coherent with }\mathcal{U}\}.\] 

The set $C_0\cap C_1$ is a club set, so there is $\alpha\in S\cap C_0\cap C_1$ such that $A_\alpha$ encodes a $\mathbb{P}_\alpha$-name for the filter $\dot{\mathcal{H}}_\alpha$. Thus, we have $\mathbb{P}_\alpha\Vdash
	\dot{\mathbb{Q}}_\alpha=SP^*(\dot{\mathcal{H}}_\alpha)$. By Lemma \ref{preservation_destruction} we have that $\mathbb{P}_{\alpha+1}$ forces that $\dot{\mathcal{H}}_{\alpha}$ cannot be extended to a P-point in any further
	$\omega^\omega$-bounding extensions, so it cannot be extended to a P-point in $V[G\ast H]$. Since $\dot{\mathcal{H}}_\alpha\subseteq\dot{\mathcal{F}}[1]\subseteq\mathcal{F}$, we have that $\mathcal{F}$ is not a P-point, which is a contradiction.

Since $\mathbb{B}_\kappa$ is ccc and adds $\kappa$ new reals, all cardinal numbers are preserved and the continuum is at least $\kappa$ in $V[G\ast H]$.

\end{proof}

\begin{cor}
It is relatively consistent with $\mathsf{ZFC}$ that there is a P-measure but there is no P-point (and so, there is no ultrafilter density that is a P-measure).
\end{cor}

We will finish this section with some remarks on the random forcing.

 As we have mentioned at the beginning of this section, the proof of Theorem \ref{mmain} could be substantially shortened if we know that the random forcing cannot add a 'new' P-point (i.e. that every P-point in the random extension extends a ground
	model P-point). However, we do not know that, nor do we know that this is not true. 

\begin{prob}\label{random_question} Is there a model $V$ such that in the random extension $V^\mathbb{B}$ there is a P-point which does not extend any P-point from $V$? 
\end{prob}

	Recall that it is still an open problem if there is a P-point in the classical random model. Kunen (see \cite{Jorg-problems}) proved that if we first add $\omega_1$ Cohen reals to the ground model
	and then add random reals, then in the resulting model there is a P-point. Dow (see \cite{Dow-Pfilters}) showed that this assertion holds also if we assume $\square_{\omega_1}$ in the ground model. In both cases, the constructed P-point is an extension of a
	P-point from the ground model.

	The proof of Theorem \ref{mmain} shows that consistently random forcings do not add 'new' P-points. Also, (the second part of) Proposition \ref{ppmeasure} says that if we add random reals to a model without P-measures, then we cannot add any  P-point. 

	So, there is a chance that Problem \ref{random_question} has a negative solution. On the other hand, if the answer is positive, the proof could bring us closer to the solution of the problem of the existence of P-point in the random model.

	Finally, notice a peculiar symmetry of Theorem \ref{p_measures_random} and Proposition \ref{ppmeasure}: every P-point in the ground model induces naturally a P-measure in the random extension, but also every P-point in the random extension
	induces in a natural way a P-measure in the ground model.

\bibliographystyle{alpha}
\bibliography{bib-ppoint}

\end{document}